\documentclass [11pt]{amsart}
\usepackage {amsmath, amssymb, amscd, graphicx, color, url,epsf,stmaryrd,enumitem}
\usepackage[all,knot,arc]{xy}
\usepackage[text={6.5in,9in},centering,letterpaper,dvips]{geometry}
\usepackage{hyphenat}
\usepackage{comment}
\usepackage{color}

\newcommand{\defin}[1]{{\bf\emph{#1}}}

\newcommand\Wr{\mathrm{wr}}
\newcommand\pGrid{G}
\newcommand\Weight{\mathbf{W}}
\newcommand\Concord{\mathcal {C}}
\newcommand\HFKt{{\rm {tHFK}}}
\newcommand\CFKt{{\rm {tCFK}}}
\newcommand\Hom{{\rm {Hom}}}
\newcommand\Ext{{\rm {Ext}}}
\newcommand\R{{\mathbb {R}}}
\newcommand\upsMin{\upsilon_{min}}
\newcommand\upsMax{\upsilon_{max}}
\newcommand\UHFK{{\mathrm{HFK}'}}
\newcommand\UHFL{{\mathrm{HFL}'}}
\newcommand\UCFK{{\mathrm{CFK}'}}
\newcommand\Grid{\mathbb{G}}
\newcommand\UGC{{\mathrm{GC}'}}

\newcommand\Gen{\mathbf{S}}
\newcommand\Xs{\mathbb{X}}
\newcommand\Os{\mathbb{O}}

\newcommand\UCFL{\mathrm{CFL}'}

\newtheorem {theorem}{Theorem}[section]
\newtheorem {lemma}[theorem]{Lemma}
\newtheorem {prop}[theorem]{Proposition}
\newtheorem {corollary}[theorem]{Corollary}

\theoremstyle{definition}
\newtheorem {definition}[theorem]{Definition}

\newtheorem {remark}[theorem]{Remark}
\newtheorem {example}[theorem]{Example}

{\makeatletter\gdef\reallynopagebreak{\nopagebreak\@nobreaktrue}}

\newenvironment{prooff}{\par \noindent {\bf Proof}\ }{\hfill
$\square$\medskip \par}
        \def\sqr#1#2{{\vcenter{\hrule height.#2pt
                \hbox{\vrule width.#2pt height#1pt \kern#1pt
                \vrule width.#2pt}\hrule height.#2pt}}}
        \def\square{\mathchoice\sqr67\sqr67\sqr{2.1}6\sqr{1.5}6}

\def\qed{~\hfill$\square$}

\newcommand\Z{\mathbb{Z}}

\newcommand\Field{\mathbb F}

\newcommand\Dual{\mathcal D}
\newcommand\Duality\Dual

\newcommand\HFKm{{\mathrm {HFK}}^-}
\newcommand\HFKa{\widehat{\mathrm{HFK}}}

\newcommand\ws{\mathbf w}
\newcommand\zs{\mathbf z}

\newcommand\relspinc{\underline{\spinc}}

\newcommand\x{\mathbf x}

\newcommand\y{\mathbf y}

\newcommand\ModSphere{\ModFlow\left({\mathbb S}\longrightarrow 
\Sym^{g-1}(\Sigma_{1})\times \Sym^2(\Sigma_{2})\right)}
\newcommand\ModSpheres\ModSphere

\newcommand\gr{\mathrm{gr}}
\newcommand\Mas{\mu}
\newcommand\UnparModSp{\widehat \ModSp}
\newcommand\UnparModFlow\UnparModSp
\newcommand\Mod\ModSp

\newcommand{\spinc}{\mathfrak s}

\newcommand\ModMaps{\mathcal M}
\newcommand\ModSp\ModMaps

\newcommand\Ta{{\mathbb T}_{\alpha}}
\newcommand\Tb{{\mathbb T}_{\beta}}

\newcommand\alphas{\mbox{\boldmath$\alpha$}}

\newcommand\betas{\mbox{\boldmath$\beta$}}

\DeclareMathOperator{\writhe}{wr}

\def\NEunnorm{\mathcal I}
\def\NESW{\mathcal J}

\newcommand{\abs}[1]{\lvert#1\rvert}

\newcommand\lk{\mathrm{lk}}

\newcommand\orL{\vec{L}}

\newcommand\HFLm{\mathrm{HFL}^-}

\newcommand\spincrel\relspinc

\newcommand\CFK{\mathrm{CFK}}

\newcommand\CFKa{\widehat\CFK}
\newcommand\CFKm{\CFK^-}

\newcommand{\Tors}{\mathrm{Tors}}

\def\endproof{\relax\ifmmode\expandafter\endproofmath\else
  \unskip\nobreak\hfil\penalty50\hskip.75em\hbox{}\nobreak\hfil\bull
  {\parfillskip=0pt \finalhyphendemerits=0 \bigbreak}\fi}
\def\endproofmath$${\eqno\bull$$\bigbreak}
\def\bull{\vbox{\hrule\hbox{\vrule\kern3pt\vbox{\kern6pt}\kern3pt\vrule}\hrule}}

\newcommand{\OneHalf}{\frac{1}{2}}

\newcommand{\Zmod}[1]{\Z/{#1}\Z}

\newcommand{\ModSWfour}{\mathcal{M}}
\newcommand{\ModFlow}{\ModSWfour}

\newcommand\abuts\Rightarrow

\begin{document}

\title{Unoriented knot Floer homology and the unoriented four-ball genus}

\author[Peter S. Ozsv\'ath]{Peter S. Ozsv\'ath}
\thanks {P. Ozsv{\'a}th was partially supported by NSF DMS-1405114}
\address {Department of Mathematics, Princeton Unversity\\ Princeton,
  New Jersey 08544}
\email {petero@math.princeton.edu}

\author[Andr{\'a}s I.  Stipsicz]{Andr{\'a}s I.  Stipsicz}
\thanks {A. Stipsicz was partially
supported by the Lend{\"u}let program ADT, ERC Advanced Grant LDTBud and
OTKA K112735}
\address {MTA R\'enyi Institute of Mathematics, Budapest, Hungary}
\email {stipsicz@renyi.hu}

\author[Zolt{\'a}n Szab{\'o}]{Zolt{\'a}n Szab{\'o}}
\thanks{Z. Szab{\'o} was partially supported by NSF DMS-1006006 and
NSF DMS-1309152}
\address{Department of Mathematics, Princeton University\\ Princeton, New Jersey 08544}
\email {szabo@math.princeton.edu}

\begin{abstract} 
  In an earlier work, we introduced a family $\HFKt (K)$ of
  $t$-modified knot Floer homologies, defined by modifying the construction of
  knot Floer homology $\HFKm$. The resulting groups 
  were then used to define concordance homomorphisms
  $\Upsilon(t)$ indexed by $t\in[0,2]$.  In the present work we
  elaborate on the special case $t=1$, and call the corresponding
  modified knot Floer homology the \emph{unoriented knot Floer
    homology} of $K$.  The corresponding concordance homomorphism when
  $t=1$ is denoted by $\upsilon$. Using elementary methods (based
  on grid diagrams and normal forms for surface cobordisms),
  we show that $\upsilon$ gives a lower bound for the smooth 4-dimensional
  crosscap number of $K$ --- the minimal first Betti number of a
  smooth (possibly non-orientable) surface in $D^4$ that meets the
  boundary $S^3$ along the given knot $K$.
\end {abstract}

\maketitle
\newcommand\Unknot{\mathcal O}
\section{Introduction}

Earlier work~\cite{Upsilon} gives a family of concordance invariants
$\Upsilon _K (t)\in {\mathbb {R}}$ ($t\in [0,2]$), associated to a
knot $K\subset S^3$. These numerical invariants are derived from the
\emph{$t$-modified knot Floer homology} $\HFKt (K)$ \cite{Upsilon},
defined using a modification of knot Floer homology (introduced in
\cite{OSKnots, RasmussenThesis}).  In \cite{Upsilon}, the following
properties of the invariants $\Upsilon _K(t)$ are verified:
\begin{enumerate}[label=($\Upsilon$-\arabic*),ref=($\Upsilon$-\arabic*)]
\item 
  \label{property:ConnSumAdd} for a connected sum $K_1\# K_2$ we have $\Upsilon _{K_1\# K_2}(t)=
\Upsilon _{K_1}(t)+\Upsilon _{K_2}(t)$;
\item 
  \label{property:LowerBound} $\Upsilon _K (t) $ provides a lower bound for the slice
genus $g_s(K)$: for $t\in [0,1]$ we have 
\[
\vert \Upsilon _K (t)\vert \leq t \cdot g_s (K);
\]
\item 
  \label{property:ConcHomomorphism}
  by combining Properties~\ref{property:ConnSumAdd}
  and~\ref{property:LowerBound}, for each $t\in [0,2]$ the map
  $K\mapsto  \Upsilon _K(t)$ provides a homomorphism from the
  smooth concordance group $\Concord$ to ${\mathbb {R}}$;
\item $\Upsilon _K (t)=\Upsilon _K (2-t)$ and $\Upsilon _K (0)=\Upsilon _K
  (2)=0$;
\item for $t=\frac{m}{n}$ the value $\Upsilon _K (t)$ is in
  $\frac{1}{n}\Z$, in particular, for $t=1$ we have that $\Upsilon
  _K(1)$ is an integer.
\end{enumerate}
Furthermore, for some classes of knots, $\Upsilon_K$ can be readily
described.  For an alternating knot $K$, $\Upsilon_K(t)$ can be
described in terms of the signature and the Alexander polynomial of
$K$.  For a torus knot $K$ (and more generally, any knot with an
$L$-space surgery) the Alexander polynomial $\Delta _K $ determines
$\Upsilon _K (t)$.  By partially computing these invariants in a
family of satellite knots, one can show that the concordance group
$\Concord$, and similarly its subgroup $\Concord _{TS}$ given by the
classes of topologically slice knots, admit a direct summand
isomorphic to $\Z ^{\infty}$, reproving a recent result of
Hom~\cite{JenHom}.

In this paper, we focus on one particular member of this family, where
$t=1$, and study how it is related to concordance problems involving
non-orientable surfaces.  The $t$-modified knot Floer homology $\HFKt
(K)$ for $t=1$ is particularly simple; it is denoted $\UHFK (K)$, and
it is called the \emph{unoriented knot Floer homology} of $K$.  The
construction is recalled in Section~\ref{sec:defs}.  By construction,
$\UHFK(K)$ is a $\Z$-graded module over the polynomial ring $\Field
[U]$.  The invariant $\upsilon (K)$ (upsilon of $K$) is defined as the
value of $\Upsilon _K(t)$ at $t=1$: this is the maximal grading of any
homogeneous, non-torsion element in the $\Field [U]$-module $\UHFK
(K)$.

We will relate $\upsilon(K)$ with the following analogue of the slice genus.
The \defin{smooth 4-dimensional crosscap
  number} $\gamma_4(K)$ of a knot $K\subset S^3$ is the minimal
\[b_1(F)=b_1(F;\Zmod{2})=\dim_{\Zmod{2}} H_1(F;\Zmod{2})\] of any
smoothly embedded (possibly non-orientable) compact surface $(F,
\partial F)$ in $(D^4, S^3)$ with $\partial F= F\cap S^3=K$. The slice
genus $g_s(K)$ is defined similarly, only there the surfaces
are required to be orientable, and we minimize the
genus (which is twice the first Betti number); so clearly $\gamma_4(K)\leq
2 g_s(K)$. The gap between these two invariants can be arbitrarily
large: for example, for $n>0$, the $(2,2n+1)$ torus knot
$T_{2,2n+1}$ has $g_s(T_{2,2n+1})=n$, but since this
(non-slice) knot can be presented as the boundary of a M\"obius band
in $S^3$, $\gamma_4(T_{2,2n+1})=1$ for all $n\in {\mathbb {N}}$.  For
more on $\gamma _4$ see \cite{GiLi}.

We wish to generalize
the slice bound from $\upsilon$ (the $t=1$ specialization of Property~\ref{property:LowerBound})
\begin{equation}\label{eq:slicebound}
\vert \upsilon (K)\vert \leq g_s (K)
\end{equation}
to a bound on $\gamma _4(K)$. This generalization involves the {\em{normal Euler number}}
of the (possibly non-orientable) surface $F$: since $F\subset D^4$,
and the ambient manifold is oriented, a non-orientable surface $F$ 
has a well-defined, integer-valued self-intersection number $e(F)$,
cf. Section~\ref{sec:knotcobordism}. 
\begin{theorem}
  \label{thm:UnorientSlice}
  Suppose that $F\subset [0,1]\times S^3$ is a (not necessarily
  orientable) smooth cobordism from the knot $K_0\subset \{ 0\} \times
  S^3$ to the knot $K_1\subset \{ 1\} \times S^3$. Then, we have
  \[\abs{\upsilon(K_0)-\upsilon(K_1) +
  \frac{e(F)}{4}}\leq \frac{b_1(F)}{2}.\]
\end{theorem}
Theorem~\ref{thm:UnorientSlice} is a direct generalization of
Equation~\eqref{eq:slicebound}: if $S$ is an orientable surface in $B^4$ meeting $S^3$ along $K$,
remove a ball centered at a point in $S$ to obain a smooth cobordism $F$ from $K$ to the unknot $\Unknot$, which has 
$\upsilon(\Unknot)=0$.
Since $F$ is orientable, $e(F)=0$ and $b_1(F)=2g(F)$. 

Theorem~\ref{thm:UnorientSlice} is
reminiscent of the ``adjunction inequalities'' pioneered by Kronheimer and
Mrowka in gauge theory~\cite{KMPolyStruct}; there, too, the genus bounds are
corrected by a self-intersection number (though the adjunction inequalities
apply to orientable surfaces).

Analogous bounds for non-orientable surfaces using a different knot
concordance invariant, $d(S^3_{-1}(K))$ of the 3-manifold 
$S^3_{-1}(K)$ given by $(-1)$-surgery along $K$, were found by
Batson~\cite{Batson} (and further generalized in~\cite{LRS}):
\[ \frac{e(F)}{2}-2d(S^3_{-1}(K))\leq b_1(F).\]
Our bounds, though, are slightly different from  these: unlike 
$d(S^3_{-1}(K))$, the invariant $\upsilon (K)$
 is additive under connected sums. 

Theorem~\ref{thm:UnorientSlice} should be compared with bounds on the crosscap number coming from the signature $\sigma(K)$
of a knot,  obtained using the Gordon-Litherland formula~\cite{GordonLitherland}:
\begin{equation}\label{eq:GL}
\vert \sigma (K)-\frac{e(F)}{2}\vert \leq b_1(F).
\end{equation}
(We use the sign convention for the signature with $\sigma
(T_{2,3})=-2$ for the right-handed trefoil knot $T_{2,3}$.)  Combining
Theorem~\ref{thm:UnorientSlice} with Equation~\eqref{eq:GL} gives:
\begin{theorem}\label{thm:UpsilonAndSignature}         
  
  For a knot $K\subset S^3$ , 
  $\abs{\upsilon(K)-\frac{\sigma(K)}{2}}\leq \gamma_4(K)$.
\end{theorem}
\begin{proof}
Suppose that $S\subset D^4$ is a smooth, compact surface with
$\partial S=K$.  Apply Theorem~\ref{thm:UnorientSlice} for the
cobordism we get from $S$ by deleting a small ball from $D^4$ centered
on $S$, we find that $\vert
\upsilon (K)-\frac{e(S)}{4}\vert \leq \frac{b_1(S)}{2}$. Combining
this with the half of Inequality~\eqref{eq:GL} we get $\vert \upsilon
(K)-\frac{\sigma (K)}{2}\vert \leq b_1(F)$, implying the desired
inequality.
\end{proof}

For knots and links in $S^3$, unoriented knot Floer homology can be
set up in several ways. We could see it as a modification of the
construction of knot Floer homology, as defined using
pseudo-holomorphic curves; or alternatively, we can define it using
grid diagrams as in~\cite{MOS,MOST}. The equivalence of the two
approaches follows from~\cite{MOS}, and the invariance proof entirely
within the grid approach is given in~\cite{MOST}, see
also~\cite{GridBook}. In this paper, we will freely use the
interchangeability of these two approaches; though, in the spirit of
Sarkar's proof of the slice bounds coming from $\tau$~\cite{Sarkar},
our proof Theorem~\ref{thm:UnorientSlice} relies mostly on grid
diagrams.


Like Sarkar's proof of the slice genus bounds for $\tau$ in \cite{Sarkar}, the
proof of Theorem~\ref{thm:UnorientSlice} uses a normal form for knot
cobordisms; for the crosscap number bound, though, we need an unorientable
version, due to Kamada~\cite{Kamada}. (The appropriately modified versions of
these results will be recalled in Section~\ref{sec:knotcobordism}.)

The invariant $\upsilon(K)$ can be computed for
many families of knots, for which the knot Floer homology is
understood. For example, following from \cite{Upsilon}, for an
alternating knot $K$ we have
  \[\upsilon(K)=\frac{\sigma (K)}{2}.\]
(Indeed, the same formula holds for the wider class of
  ``quasi-alternating knots'' of~\cite{BrDCov}.)

  We can also describe $\upsilon$ for the torus knot $T_{p,q}$. To this end,
  write the symmetrized Alexander polynomial $\Delta _{T_{p,q}}(t)$ of
  $T_{p,q}$ as
\[ \Delta_{T_{p,q}}(t)=
\frac{(t^{pq}-1)(t-1)}{(t^p-1)(t^q-1)} t^{-(\frac{pq-p -q-1}{2})}=
\sum_{k=0}^n (-1)^k t^{\alpha_k}, \]
where $\alpha_i$ is a decreasing sequence of integers.
Define a corresponding sequence of numbers inductively by
\begin{align*}
m_{0}&=0 \\
m_{2k} &= m_{2k-1}-1 \\
m_{2k+1} &= m_{2k}-2(\alpha_{2k}-\alpha_{2k+1})+1.
\end{align*}
(Recall from~\cite{NoteLens} that $\HFKa(T_{p,q})$ consists of the
direct sum of $\Field = \Z/2\Z$ summands
supported in bigradings $\{(m_{k},\alpha_{k})\}_{k=0}^n$, where $m_k$
denotes the Maslov and $\alpha_k$ the Alexander gradings.)
As a specialization of the computation of $\Upsilon _K (t)$ for 
torus knots~\cite[Theorem~\ref{Concordance:thm:TorusKnots}]{Upsilon}, we get

\begin{theorem}
  \label{thm:TorusKnots}
  For the positive $(p,q)$ torus knot $T_{p,q}$,
$\upsilon(T_{p,q})=\max_{0\leq 2k\leq n} \{ m_{2k}-\alpha_{2k}\}$.
\qed
\end{theorem}
More generally, Theorem~\ref{thm:TorusKnots} holds
for any knot in $S^3$ for which some positive rational surgery gives an 
``$L$-space'' in the
sense of~\cite{NoteLens}. Torus knots have
this property; and other knots (e.g. certain iterated torus knots) also
satisfy this condition.  

For example, for the torus knot $T_{3,4}$ we have $\Delta _{T_{3,4}}(t)=
t^3-t^2+1- t^{-2} + t^{-3}$, and so $\upsilon (T_{3,4})=-2$.
Since $\sigma (T_{3,4})=-6$, Theorem~\ref{thm:UpsilonAndSignature}
implies that $\gamma _4 (T_{3,4})\geq 1$. Since the knot $T_{3,4}$ can
be presented as the boundary of a M\"obius band
(cf. \cite[Figure~4.1]{Batson}), we actually get that $\gamma _4
(T_{3,4})=1$. 
On the other hand, the additivity of both $\upsilon$ and
$\sigma$, together with the above calculation provides
\begin{corollary}
  Consider the knot $K_n=\# _n T_{3,4}$, the $n$-fold connected sum of
  $T_{3,4}$. Then $\upsilon (K_n)=-2n$ and $\sigma (K_n)=-6n$, therefore
  $\gamma _4 (K_n)=n$.  \qed
\end{corollary}
Note that this observation reproves \cite[Theorem~2]{Batson} of
Batson, showing that the 4-dimensional smooth crosscap number $\gamma
_4$ can be arbitrarily large.

The $t=1$ specialization of Property~\ref{property:ConcHomomorphism} shows that $\upsilon$ induces
a homomorphism from the smooth concordance group to $\Z$.
One might wonder about the relationship between $\upsilon$ and previously existing
concordance homomorphisms.
Infinitely many linearly independent homomorphisms from the smooth
concordance group to $\Z$ were constructed in work of Jen
Hom~\cite{JenHom}; but previous to this work, there were a few other concordance
homomorphisms that are non-trivial on topologically slice knots. For example,
there is $\tau(K)$, $\delta(K)$ (the $d$ invariant of the double
branched cover of $S^3$ along $K$, studied by Manolescu and
Owens~\cite{ManolescuOwens}); and Rasmussen defined an invariant
$s(K)$ using Khovanov homology.  Computing these invariants on
appropriate examples quickly leads to the following independence
result:
\begin{prop}
  \label{prop:Independence}
  The homomorphism $\upsilon$ is linearly independent from $\tau$,
  $\delta$, $s$, and $\sigma$.
\end{prop}

The genus bounds obtained here are similar to earlier results; for
example, those of~\cite{Batson} and~\cite{LRS} in the non-orientable
case and~\cite{FourBall} and~\cite{RasmussenThesis} in the orientable
case. Those proofs rely on the Heegard Floer homology groups for
closed three-manifolds, and how these groups are related under
cobordisms.  By contrast, our present work relies on cominatorial
decompositions of (possibly unorientable) knot cobordisms, in the
spirit of Sarkar~\cite{Sarkar} (for slice genus bounds using $\tau$)
and the earlier work of Rasmussen~\cite{RasmussenSlice} (for slice genus bounds using  Khovanov
homology).

The paper is organized as follows.  In Section~\ref{sec:defs} we
provide the definition of unoriented knot Floer homology, both from
the holomorphic and from the grid theoretic point of view. Since the
definition relies on constructions discussed in detail elsewhere, we
will make frequent references to those sources.  Indeed, since $\UHFK$
is a special case of the $t$-modified knot Floer homolog $\HFKt$,
basic properties of unoriented knot Floer homology follow from general
discussions of \cite{Upsilon}. We define also related invariant for links,
which will be needed later.
In Section~\ref{sec:unknotting} we
verify a bound on the change of $\upsilon$ under crossing changes.
Although the result of Section~\ref{sec:unknotting} also follows from
results of \cite{Upsilon}, we devoted this section to describe a more
direct proof.  In Section~\ref{sec:knotcobordism} we review what is
needed about (orientable and non-orientable) cobordisms between
knots. In particular, we quote the necessary normal form theorems.  In
Section~\ref{sec:slicebounds} we give the details of the bounds on the
genera and Betti numbers (in the orientable and in the non-orientable
case) provided by the $\upsilon$-invariant.  Although the oriented
case already follows from~\cite{Upsilon}, we give an alternate
combinatorial proof, which is then easily modified to apply in the
non-orientable case, as well.  
In Section~\ref{sec:computations} we
give a few sample computations of $\UHFK (K)$ and $\upsilon (K)$.
In Section~\ref{sec:Orientations}, we give a small modification of the
earlier link invariant, to define unoriented link invariants.

{\bf Acknowledgements:} We would like to thank
Josh Batson, Ciprian Manolescu and Sucharit  Sarkar for useful
discussions.

\newcommand\CFm{{\mathcal{CF}}^-}
\newcommand\Cone{\mathrm{Cone}}
\section{Definition of $\upsilon$}
\label{sec:defs}

We start by recalling the definition of unoriented knot Floer homology
$\UHFK (K)$.  Although the invariant has been described in
\cite{Upsilon} (as $\HFKt (K)$ with $t=1$), for completeness (and
since some of the constructions are needed in our later arguments) we
give the details of the definition here.  We start our discussion in
the holomorphic context, and will turn to grid diagrams afterwards.

\subsection{Unoriented knot Floer homology}
Let ${\mathcal {H}}=(\Sigma,\alphas,\betas,w,z)$ be a genus-$g$ doubly
pointed Heegaard diagram for a knot $K\subset S^3$.  Let $\Gen
({\mathcal {H})}$ denote the set of Heegaard Floer states for
the diagram, that is, $\Gen ({\mathcal {H})}$ is the set of
unordered $g$-tuples $\x =\{ x_1, \ldots , x_g\}\subset \Sigma$ such
that each $\alpha _i \in \alphas$ and each $\beta _j \in \betas$
contains a unique element of $\x$.  There are maps $M 
\colon \Gen ({\mathcal {H})} \to
\Z$ (the ``Maslov grading'') and $A\colon \Gen ({\mathcal {H})}
\to\Z$ (the ``Alexander
grading''). For the definitions and detailed discussions of these
notions, see \cite{OSKnots}; explicit formulae will be given only in
the grid context.

Define the $\Z$-grading of the
state $\x$ by the difference
\[ 
\delta(\x)=M(\x)-A(\x).
\]

Consider the $\Field [U]$-module $\UCFK ({\mathcal {H}})$ freely generated
by the Heegaard Floer states. We extend the $\Z$-grading
by defining
\[ \delta(U^i\cdot \x)=\delta(\x)-i.\]
(Note that this convention is compatible with the usual conventions,
since multiplication by $U$ drops the Maslov grading $M$ by 2 and the
Alexander grading $A$ by 1.) Equip $\UCFK ({\mathcal {H}})$ with the
modified Heegaard Floer differential
\begin{equation}\label{eq:hatar}
\partial \x = \sum_{\y\in\Gen ({\mathcal {H})}}\sum _{\{\phi
  \in \pi _2 (\x , \y ) \mid \mu (\phi )=1\}}
\#\left(\frac{\ModFlow(\phi)}{{\mathbb R}}\right)
U^{n_{w}(\phi)+n_{z}(\phi)} \y,
\end{equation}
where $\mu (\phi )$ is the Maslov index (formal dimension) of the moduli space
$\ModFlow (\phi )$ of holomorphic disks representing $\phi \in \pi _2 (\x, \y
)$, and $n_w(\phi)$ (and similarly $n_z(\phi )$) is the multiplicity of the
domain corresponding to $\phi $ at $w$ (and $z$, resp.). The symbol
$\#\left(\frac{\ModFlow(\phi)}{{\mathbb R}}\right)$ denotes the mod 2 count of
elements in the quotient of the moduli space (with $\mu (\phi )=1$) by the
obvious ${\mathbb {R}}$-action.  For the moduli space $\ModFlow (\phi )$ to
make sense, one needs to fix an almost complex structure on the appropriate
symmetric power of the Heegaard surface --- for more details see~\cite{OSzI}.

\begin{definition}
  The homology of $(\UCFK ({\mathcal {H}}), \partial )$ is called the
  \defin{unoriented knot Floer homology} of the knot $K\subset S^3$,
  and will be denoted by $\UHFK(K)$.
\end{definition}

In~\cite{Upsilon}, we give a more general construction, parameterized by a parameter $t$.
The chain complex $\CFKt$ is given a grading where $\gr_t(\x)=M(\x)-t A(\x)$, and the differential is computed by
\[
  \partial_t \x = \sum_{\y\in\Gen}\sum_{\{\phi\in\pi_2(\x,\y)\big|\Mas(\phi)=1\}} \# \left(\frac{\ModFlow(\phi)}{\mathbb R}\right)
  U^{t n_z(\phi)+ (2-t) n_w(\phi)} \y.
\]
Setting $t=1$ in this construction gives back unoriented knot complex $\UCFK$, with the $\Z$-grading induced by $\delta$.
Since the homology $\HFKt$ of $\CFKt$ is a knot invariant, so is the $t=1$ specialization:
\begin{theorem}$($\cite[Theorem~1.1]{Upsilon}$)$
  \label{thm:UHFKwellDefined}
  The homology $\UHFK(K)$, as a $\Z$-graded $\Field [U]$-module,
  is an invariant of $K$. \qed
\end{theorem}


\begin{remark}
  The $t$-modified knot Floer homology $\HFKt$ is defined for all $t\in [0,2]$,
  and in the generic case we need to use a more complicated base ring,
  the ring of ``long power series'' (cf. \cite[Section~11]{Brandal}).
  For rational $t$ (and in particular, for $t=1$), however, appropriate 
  polynomial rings are also sufficient, as it is applied in the above definition;
  see~\cite[Proposition~\ref{Concordance:prop:SameUpsilon}]{Upsilon}.
\end{remark}
In the usual setting of knot Floer homology, by setting $U=0$ in
the chain complex $\CFKm$, and then taking homology,
we get a related, simpler invariant, denoted
$\HFKa$.

\begin{prop}
  The homology of $\UCFK( {\mathcal {H}})/(U=0)$ is isomorphic to
  $\HFKa(K)$ (when in the latter group we collapse the Maslov and
  Alexander gradings to $\delta = M-A$).
\end{prop}

\begin{proof}
 By setting $U=0$, the differentials
 for both $\UCFK/U$ and $\CFKa$ count those holomorphic disks 
 for which both $n_w$ and $n_z$ vanish, hence the
 resulting chain complexes are isomorphic. The isomorphism obviously
 respects the grading $\delta=M-A$.
\end{proof}

Let
$\upsilon(K)$ be the  maximal $\delta$-grading of any 
homogeneous   non-torsion element in $\UHFK(K)$:
\[
\upsilon (K)=\max \{ \delta (x)\mid x\in \UHFK (K) \quad \text{homogeneous
and} \quad U^d\cdot x\neq 0 \quad \text{for all} \quad d\in {\mathbb {N}}\}.
\]
Since $\UHFK(K)$ is bounded above, to see that the above definition
makes sense, we must show that there are non-torsion elements in
$\UHFK$.  This could be done by appealing to the holomorphic theory;
alternatively, we can appeal to Proposition~\ref{prop:StructureHFKinf}
proved below.  Assuming this, Theorem~\ref{thm:UHFKwellDefined}
immediately implies that $\upsilon(K)$ is a knot invariant.  In fact,
$\upsilon(K)=\Upsilon_K(1)$, in the notation of~\cite{Upsilon}.

\begin{remark}
In the choice of the sign of $\upsilon$ we follow the convention 
of \cite{Upsilon} (in particular, $\upsilon (K)=\Upsilon _K(1)$).
This convention differs from the convention for the $\tau$-invariant, where 
we have 
\[
\tau (K)=-\max \{ A(x)\mid x\in \HFKm (K) \quad \text{homogeneous and}
\quad U^d \cdot x \neq 0 \quad \text{for all} \quad d\in {\mathbb {N}} \}.
\]
\end{remark}

\subsection{Formal constructions}
\label{subsec:Formal}

Let $({\mathcal C},\partial)$ be a $\Z$-graded free chain complex over
$\Field[U]$ with a $\Z$-valued filtration, with the compatibility
conditions that multiplication by $U$ drops grading by two and
filtration level by one.  Let $\Gen$ be a homogeneous generating set
for ${\mathcal C}$ over $\Field[U]$; so there are functions $M\colon
\Gen\to \Z$ and $A\colon \Gen\to\Z$ so that the element $\x\in\Gen$ is
in grading $M(\x)$, and filtration level $A(\x)$.  We can form another
complex $(C',\partial')$ with a $\Z$-grading by the following
construction.  $C'$ is also generated by $\Gen$, its $\Z$-grading is
induced by $\delta=M-A$.  The differential on $C'$ is specified by the
property that $U^m \y$ appears with coefficient $1$ in the
differential $\partial \x$  for $\x,\y\in\Gen$ (so that
$m=\frac{M(\y)-M(\x)}{2}$) if and only if
$U^{\frac{\delta(\y)-\delta(\x)+1}{2}}\cdot \y$ appears with
coefficient $1$ in $\partial' \x$.

For example, a knot $K\subset S^3$ induces a filtration on
$\CFm(S^3)$; if ${\mathcal C}$ denotes the resulting filtered chain
complex, then it is straightforward to check that $C'$ coincides with
the construction of $\UCFK(K)$ from above.
(See~\cite[Section~\ref{Concordance:sec:Formal}]{Upsilon} for the
generalization of this construction for $t\in[0,2]$.)

Constructions from knot Floer homology can be easily lifted to
constructions to unoriented knot Floer homology, using the above
formal trick. For example, if $({\mathcal C}_1,\partial_1)$ and $({\mathcal C}_2,\partial_2)$,
are two $\Z$-filtered, $\Z$-graded  free chain
complexes over $\Field[U]$, and $\phi\colon {\mathcal C}_1\to{\mathcal C}_2$ is a homotopy equivalence between them, then $\phi$
induces a homotopy equivalence
$\phi'\colon C'_1\to C'_2$  between their corresponding formal
constructions $(C_1',\partial_1')$ and $(C_2',\partial_2')$. This is how Theorem~\ref{thm:UHFKwellDefined} is derived from 
the invariance of the filtered chain homotopy type of $\CFm(S^3)$ with its induced filtration from $K$;
see~\cite[Theorem~1.1]{Upsilon}.

\subsection{Multi-pointed diagrams}
Like knot Floer homology, unoriented knot Floer homology can be
computed using Heegaard diagrams with multiple basepoints:
\begin{definition}\label{def:weights}
  Let ${\mathcal
    {H}}=(\Sigma,\alphas,\betas,\{w_1,\dots,w_n\},\{z_1,\dots,z_n\})$
  be a multi-pointed Heegaard diagram for $K\subset S^3$.  Given
  $\phi\in\pi_2(\x,\y)$, define its weight as
  \[ \Weight(\phi) = \sum_{i=1}^n n_{w_i}(\phi )+n_{z_i}(\phi).\]
Consider the free $\Field [U]$-module $\UCFK({\mathcal {H}}) $
generated by the Heegaard Floer states of the Heegaard diagram
${\mathcal {H}}$, and define the boundary map as
  \[ \partial \x = \sum_{\y\in\Gen ({\mathcal {H})}}\sum _{\{ \phi \in \pi _2(\x, \y) \mid 
\mu (\phi )=1\} } \#\left(\frac{\ModFlow(\phi)}{{\mathbb R}}\right)
  U^{\Weight(\phi)} \y. \]
The $\delta$-grading (as the difference $M-A$ of the Maslov and
Alexander gradings) extends naturally to the multi-pointed
setting.
\end{definition}

For the next theorem, it is convenient to introduce some notation. Let $V$ be the two-dimensional
$\Field$-vector space supported in $\delta$-grading equal to zero, so that if $M$ is any $\Z$-graded $\Field[U]$-module,
there is an isomorphism of $\Z$-graded $\Field[U]$-modules:
\[ M\otimes_{\Field} V \cong M \oplus M.\]
\begin{theorem}\label{thm:stabilized}
  The homology of $(\UCFK (\Sigma,\alphas,\betas,\ws,\zs), \partial )$
  is isomorphic to $\UHFK(K)\otimes _{\Field} V^{n-1}$.
\end{theorem}
\begin{proof}
  There is a model for Heegaard Floer homology with multiple
  basepoints; see~\cite{OSLinks}, and~\cite{MOS} for the case of
  knots.  In this model, the chain complex ${\mathcal
      {C}=}\CFm ({\mathcal {H}})$ for $\CFm(S^3)$ (with its filtration
  coming from $K$) is specified as a module over
  $\Field[U_1,\dots,U_n]$, with differential
  \[ \partial \x = \sum_{\y\in\Gen ({\mathcal {H})}}\sum _{\{ \phi \in \pi _2(\x, \y) \mid 
    \mu (\phi )=1\} } \#\left(\frac{\ModFlow(\phi)}{{\mathbb
      R}}\right) U_1^{n_{w_1}(\phi)}\cdots U_n^{n_{w_n}(\phi)}\cdot
  \y. \] Setting all the $U_i$ equal to one another (and denoting the
  resulting formal variable by $U$), we obtain the complex
  $\frac{{\mathcal C}}{U_1=\dots=U_n}$, a $\Z$-filtered, $\Z$-graded
  chain complex over $\Field[U]$, with differential given by
  \[ \partial \x = \sum_{\y\in\Gen ({\mathcal {H})}}
\sum _{\{ \phi \in \pi _2(\x, \y) \mid 
    \mu (\phi )=1\} } \#\left(\frac{\ModFlow(\phi)}{{\mathbb
      R}}\right) U^{n_{w_1}(\phi)+\dots+n_{w_n}(\phi)}\cdot \y; \]

  Assume for notational simplicity that $n=2$ in ${\mathcal{H}}$. In
  this case, we can destabilize the diagram after handleslides, to
  obtain a Heegaard diagram ${\mathcal{H}'}$ for $K$ with only two
  basepoints $w_1$ and $z_1$. Thus, the complex ${\mathcal
    C}=\CFm({\mathcal{H}'})$ is a filtered chain complex over
  $\Field[U_1]$.  We can promote this to a complex ${\mathcal C}[U_2]$
  over $\Field[U_1,U_2]$, and take the filtered mapping cone of
  the map
  \[ 
  U_1-U_2\colon {\mathcal C}[U_2]\to {\mathcal C}[U_2].\] As
  in~\cite[Proposition~6.5]{OSLinks} or~\cite[Theorem~1.1]{MOS}, the
  Heegaard moves induce a filtered homotopy equivalence of filtered
  complexes over $\Field[U_1,U_2]$ between the above mapping cone and
  $\CFm({\mathcal{H}})$.  Filtrations and gradings on the mapping cone
  are modified as follows. If $M$ is a $\Z$-graded $\Field[U]$ module, let 
  $M\llbracket k\rrbracket$ denote the same $\Field[U]$-module, but with grading specified by
  \begin{equation}
    \label{eq:DefShift}
    M\llbracket k\rrbracket_d = M_{k+d}.
  \end{equation}
  With this notation, the mapping cone of $U_1-U_2$ is identified with two copies
  of ${\mathcal C}[U_2]$; in fact, there is a $\Z$-graded isomorphism of $\Field[U]$-modules
  \[ \Cone(U_1-U_2)\cong {\mathcal C}[U_2]\llbracket 1\rrbracket \oplus {\mathcal C}[U_2],\]
  where the first summand represents the domain of $U_1-U_2$
  and the second its range.  Alexader gradings are shifted similarly.

  In particular, setting $U_1=U_2$, we obtain a filtered homotopy
  equivalence 
  \[ \CFm({\mathcal{H}})\simeq {\mathcal C}\otimes _{\Field} {\mathcal V}\] 
of $\Z$-filtered, $\Z$-graded modules over $\Field[U]$, where
${\mathcal V}$ is a two-dimensional $\Z\oplus\Z$-graded vector space,
with one genertor in bigrading $(0,0)$ and another in bigrading
$(-1,-1)$ (one of these components gives the $\Z$-grading and the
other the $\Z$-filtration). It follows now that
  \[ (\CFm({\mathcal{H}}))'\simeq ({\mathcal C}\otimes _{\Field}
{\mathcal V})'\cong {\mathcal C}'\otimes _{\Field} V.\]
  
  The case of arbitrary $n$ is obtained by iterating the above.
\end{proof}

\subsection{Unoriented grid homology}
It follows from Theorem~\ref{thm:stabilized} that (a suitably
stabilized version of) $\UHFK(K)$ can be computed using grid
diagrams. Explicitly, following the notation
from~\cite{MOST,GridBook}, let $\Grid$ be a grid diagram for $K$ with
markings $\Xs$ and $\Os$.  Let $\Gen (\Grid )$ denote the grid states
of $\Grid$, i.e. the Heegaard Floer states of the Heegaard diagram
induced by the the grid $\Grid$.  In this picture the Maslov and
Alexander gradings can be given by rather explicit formulae, as we
recall below.

By considering a fundamental domain in the plane
${\mathbb {R}} ^2$ for the grid torus, the $\Xs$- and $\Os$-markings
provide the values $M_{\Os}(\x )$ and $M_{\Xs}(\x )$ for a grid state $\x$, as
follows: For two finite sets $P,Q\subset \R ^2$ define
$\NEunnorm(P,Q)$ to be the number of pairs $(p_1,p_2)\in P$ and
$(q_1,q_2)\in Q$ with $p_1<q_1$ and $p_2<q_2$. Introduce the
corresponding symmetrized function
\[
\NESW(P,Q)=\frac{\NEunnorm(P,Q)+\NEunnorm(Q,P)}{2}.
\]
We view $\NESW$ as a bilinear form, so that the expression
$\NESW (P-Q,P-Q)$ is defined to mean
$\NESW (P,P)-2\NESW (P,Q)+\NESW (Q,Q)$.

With this notation in place,  consider the function
$M_{\Os}(\x)$ on the grid state $\x$ defined by 
\begin{equation}\label{eq:MaslovDef}
M_{\Os}(\x)=\NESW (\x -\Os, \x -\Os) +1;
\end{equation}
by replacing $\Os$ with $\Xs$ we get $M_{\Xs}(\x )=
\NESW (\x -\Xs, \x -\Xs)+1$. As it was
verified in \cite{MOST},
these quantities
are independent from the choice of the fundamental domain and are
functions of the grid states. Indeed, the Maslov grading of $\x$ in
the knot Floer chain complex corresponding to the grid $\Grid$ 
is equal to $M_{\Os}(\x )$, while the
Alexander grading of $\x$ is equal to
\[
A(\x )=\frac{1}{2}(M_{\Os}(\x ) -M_{\Xs}(\x ))-\frac{n-1}{2},
\]
where $n$ is the size (the grid index) of $\Grid$.  In this setting
the $\delta$-grading $\delta(\x)= M(\x)-A(\x)$ can be given as
\begin{equation}\label{eq:DeltaGrading}
\delta (\x )=\frac{1}{2}(M_{\Os}(\x ) + M_{\Xs}(\x ))+\frac{n-1}{2}.
\end{equation}

The set ${\rm {Rect}}(\x , \y )$ of rectangles from $\x$ to $\y$ is
defined in \cite{MOST}.  For a rectangle $r\in{\rm {Rect}}(\x,\y)$ let
$\Weight(r)=\# r\cap (\Xs\cup \Os)$ be the corresponding weight (as in
Definition~\ref{def:weights}).  Consider the chain complex
$\UGC(\Grid)$ freely generated over $\Field [U]$ by the grid states,
endowed with the $\delta$-grading of Equation~\eqref{eq:DeltaGrading}
and the differential
\[ 
\partial \x = \sum_{\y\in\Gen (\Grid )} \sum_{r\in {\rm
    {Rect}}^0(\x,\y)} U^{\Weight(r)} \y ,
\] 
where ${\rm {Rect}}^0 (\x, \y)$ is the set of \emph{empty} rectangles
connecting $\x $ and $\y$ (i.e. such rectangles which do not contain
in their interior any component of $\x$ or $\y$).  From
Theorem~\ref{thm:stabilized} and the identification of the moduli
space count of the holomorphic theory with counting empty rectangles
in $\Grid$ (as shown in \cite{MOS}), it follows:
\begin{corollary}
  \label{cor:ComputeUGHFromGrid}
  If $\Grid$ is a grid diagram for the knot $K$ of grid index $n$,
  then there is a $\delta$-graded $\Field [U]$-module isomorphism
  \[ H_*(\UGC(\Grid))\cong\UHFK(K)\otimes _{\Field}V^{n-1}, \]
  where $V$ is the two-dimensional $\Field$-vector space supported in
  $\delta$-grading equal to zero. \qed
\end{corollary}

%
%

Using grid diagrams, Corollary~\ref{cor:ComputeUGHFromGrid} (and the
$\delta$-grading of $V$) gives a combinatorial description of
$\upsilon(K)$.  
\begin{theorem}
  \label{thm:ComputeUpsilonFromGrid}
  The knot invariant $\upsilon(K)$ can be computed from a grid diagram
  $\Grid$ of the knot $K$: it  is the
  maximal $\delta$-grading of any non-torsion
  homogeneous element of $H_*(\UGC(\Grid))$. \qed
\end{theorem}

In fact, one can prove that the quantity defined in
the grid context is a knot invariant without appealing to the
holomorphic theory, but working entirely within the context of grid
diagrams. Setting this up is a straightforward adaptation
of the results of~\cite{MOST}.

\subsection{The case of links}
In our subsequent arguments we will need a slight extension of
$\upsilon$ for links.  Note that $t$-modified knot Floer homology
admits a straightforward extension to links
(cf. \cite[Section~10]{Upsilon}), where we use the collapsed link
Floer homology $c\HFLm $, which is a bigraded module over $\Field
[U]$.  

In more detail, recall that a link $L=(L_1, \ldots , L_{\ell})$ of
$\ell$ components in $S^3$, equipped with an orientation $\orL$, can
be represented by a multi-pointed Heegaard diagram $ {\mathcal
    {H}}=(\Sigma , \alphas, \betas , \{w_1, z_1 \}, \ldots , \{
w_{\ell }, z_{\ell }\} )$, where the pair $\{ w_i, z_i \}$ determines
the $i^{th}$ component $L_i$. In the generalization of $\HFKm$ to
$\HFLm$, a vector $(A_1 (\x ), \ldots , A_{\ell}(\x) )$ of $\ell$
Alexander gradings is associated to each generator $\x\in\Ta\cap\Tb$
(see \cite{OSLinks}), and the homology has the structure of a module
over the ring $\Field [U_1, \ldots , U_{\ell}]$.  Consider next the
chain complex $\UCFL({\mathcal {H}})$ freely generated over
$\Field[U]$ by grid states, equipped with the differential given by
\[
\partial' \x = \sum_{\y\in\Gen ({\mathcal {H})} }
\sum _{\{\phi \in \pi _2 (\x , \y ) \mid
  \mu (\phi )=1\}} \#\left(\frac{\ModFlow(\phi)}{{\mathbb R}}\right)
U^{\Weight(\phi)} \y,
\]
where $\Weight(\phi) =\sum _{i}n_{w_i}(\phi )+ \sum _{i}n_{z_i}(\phi
)$ (as in Definition~\ref{def:weights}).  Equip $\UCFL({\mathcal
  {H}})$ with the $\Z$-grading $\delta (\x )=M(\x )-A(\x )$, where
$A(\x)=\sum_{i=1}^\ell A_i(\x)$ gives another integer-valued grading.
By properties of the Maslov index (see for
example~\cite[Proposition~4.1]{OSLinks}) and the Alexander grading
(see~\cite[Lemma~3.11]{OSLinks}), it follows that for any
$\x,\y\in\Gen ({\mathcal {H})}$ and $\phi\in\pi_2(\x,\y)$,
\begin{equation}
  \label{eq:RelativeGrading}
  \delta(\x)-\delta(\y)=\Mas(\phi)-\Weight(\phi); 
\end{equation}
and so the differential on $\UCFL({\mathcal{H}})$ drops the
$\Z$-grading by one.

\begin{remark}
  For links, there are several possible choices of Maslov
  grading. We use here the Maslov grading from~\cite{OSLinks}, that is
  characterized by the property that the homology of the Heegaard
  Floer chain complex associated to ${\mathcal
      {H}_{S^3}}=(\Sigma,\alphas,\betas,\{w_1,\dots,w_\ell\})$, which
  is isomorphic to $\Field[U]$, has generator in Maslov grading equal
  to $0$.
\end{remark}

\begin{definition}
  Let
  ${\mathcal{H}}=(\Sigma,\alphas,\betas,\{w_1,z_1\},\dots,\{w_\ell,z_\ell\})$
  be a Heegaard diagram representing an oriented link $\orL$. The
  homology of the chain complex $( \UCFL ({\mathcal {H}}), \partial )$
  defined above (together with the $\delta$-grading) gives the
  \defin{unoriented link Floer homology} $\UHFL(\orL)$ of $\orL$.
\end{definition}
\begin{theorem}
  The unoriented link Floer homology $\UHFL (\orL)$, as a $\delta$-graded
  $\Field[U]$-module, is an invariant of the oriented link $\orL$. 
\end{theorem}

\begin{proof}
  Start from the filtered link complex from~\cite{OSLinks}, and set
  set variables $U_1=\dots=U_\ell$ to obtain a $\Z$-graded,
  $\Z$-filtered chain complex. According to~\cite{OSLinks}, the
  filtered chain homotopy type of this complex is a link
  invariant. Applying the formal construction from
  Section~\ref{subsec:Formal}, we arrive at the chain complex
  $\UCFL$. As it is explained in
  \cite[Section~\ref{Concordance:sec:Links}]{Upsilon}, the application
  of the formal construction producing $\HFKt (K)$ from the filtered
  knot Floer complex of $K$ applies to the above chain complex over
  $\Field [U]$, ultimately showing that the unoriented link Floer
  homology $\UHFL (\orL)$ of a link $L$ is an invariant of $\orL$.
\end{proof}

\begin{remark}
  In a Heegaard diagram
  $(\Sigma,\alphas,\betas,\{w_1,z_1\},\dots,\{w_\ell,z_\ell\})$, the
  orientation on $\orL$ is specified by choosing the labeling of the
  basepoints as $w_i$ or $z_i$.  The weight $\Weight(\phi)$ is independent of
  this choice, so the differential $\partial$ is independent of the
  orientation on $L$; and so, in view of Equation~\eqref{eq:RelativeGrading},
  $\UHFL(\orL)$, thought of as a relatively $\Z$-graded module over
  $\Field[U]$, is independent of the chosen orientation on $L$.  {The
    dependence of the the absolutely $\Z$-graded object will be described in
    Proposition~\ref{prop:UnorientedLinkInvariant}.}
\end{remark}

Grid diagrams can be used to compute unoriented link Floer homology,
as well.  We define the Alexander grading for an $\ell$-component
oriented link $L$ by
\begin{equation}
    \label{eq:DefAlexanderLinks}
    A(\x)=\frac{1}{2}(M_{\Os}(\x)-M_{\Xs}(\x)) -
    \left(\frac{n-\ell}{2}\right) \in \Z,
\end{equation}
hence the $\delta$-grading of a grid state $\x$ is equal to
\[
\delta (\x )=\frac{1}{2}(M_{\Os}(\x )+M_{\Xs}(\x))+\left(\frac{n-\ell}{2}\right).
\]

With this understanding, $\UGC(\Grid)$ can be defined for a grid
diagram $\Grid$ representing the oriented link $\orL$. For an
$\ell$-component link the homology $H_*(\UGC (\Grid ))$ is isomorphic
to $\UHFL (\orL)\otimes _{\Field} V^{n-\ell}$. Indeed, the same
handle sliding/destabilizing argument applies as in the proof of
Theorem~\ref{thm:stabilized} until we get a Heegaard diagram with
$\ell$ pairs of basepoints.

If $\orL$ is an oriented link, let ${\mathcal U}_m(\orL)$ denote the disjoint union of $\orL$ with the $m$-component unlink.

Let $W$ be the two-dimensional, $\Z$-graded vector space with one basis vector with degree $0$ and the other with degree $-1$,
so that
if $M$ is a $\Z$-graded
module over $\Field[U]$, there is an isomorphism 
\[ M\otimes_{\Field} W\cong M\oplus M\llbracket 1\rrbracket\]
of $\Z$-graded modules over $\Field[U]$, using notation from
Equation~\eqref{eq:DefShift}.

\begin{prop}
  \label{prop:KunnethDisjointUnion}
  Let $\orL$ be an oriented link with $\ell$ components.
  Then, there is an isomorphism of $\Z$-graded modules over $\Field[U]$:
  \[ \UHFL({\mathcal U}_m(\orL))\cong 
  \UHFL(\orL)\otimes_{\Field} W^{m}. \]
\end{prop}
\begin{proof}
  Consider $m=1$, and let ${\mathcal{H}}$ be an $\ell$-pointed
  Heegaard diagram for the $\ell$-component link $\orL$. An 
    $(\ell +1)$-pointed Heegaard diagram for ${\mathcal U}_1(\orL)$
  is obtained by forming the connected sum ${\mathcal{H}'}$ of
  ${\mathcal{H}}$ with a standard diagram in $S^2$, consisting of two
  embedded circles $\alpha_{\ell+1}$ and $\beta_{\ell+1}$ that
  intersect transversally in two points, dividing $S^2$ into four
  regions.  One of the regions contains the two basepoints
  $w_{\ell+1}$ and $z_{\ell+1}$, its two adjacent regions are
  unmarked, and the fourth region is used as the connected sum
  point. This is the picture for an index $0$ and $3$ stabilization as
  in~\cite[Proposition~6.5]{OSLinks}. It is similar to stabilization
  on a knot as in Theorem~\ref{thm:stabilized}, except for the
  placement of the $z$ basepoints.  Thus, the stabilization proof once
  again identifies $\CFm({\mathcal{H}'})$ with the mapping cone of
    \[ U_{\ell+1}-U_\ell \colon \CFm({\mathcal H})\to\CFm({\mathcal
      H}),\] except that the filtration conventions are different.
    The two summands correspond to the two intersection points $x$ and
    $y$ of $\alpha_{\ell+1}$ and $\beta_{\ell+1}$.  These two summands
    now have the same Alexander filtration levels (although their
    Maslov gradings are shifted as before). Thus, when we set
    $U_{\ell+1}=U_{\ell}$ in this complex, we obtain a filtered
    homotopy equivalence
  \[ \CFm({\mathcal{H}})\simeq {\mathcal C}\otimes _{\Field} {\mathcal W}\] 
of $\Z$-filtered, $\Z$-graded modules over $\Field[U]$, where
${\mathcal W}$ is a two-dimensional $\Z\oplus\Z$-graded vector space,
with one genertor in bigrading $(0,0)$ and another in bigrading
$(-1,0)$. (Again, the first component is the Maslov grading and the
second induces the Alexander filtration.)  This translates into a
$\Z$-graded quasi-isomorphism of chain complexes over $\Field[U]$:
\[ 
\UCFK({\mathcal H}')=\CFm({\mathcal{H'}})'\simeq
(\CFm({\mathcal{H}})\otimes _{\Field}{\mathcal W})'\cong \UCFK({\mathcal
  H})\otimes _{\Field} W.
\]

Iterating the above result, we arrive at the proposition for arbitrary
$m$.
\end{proof}

\begin{corollary}\label{cor:unlink}
  If $\orL$ is the $n$-component unlink, then 
  $\UHFL(\orL)\cong \Field[U]_{(0)} \otimes W^{n-1}$,
  where  $W=\Field_{(0)}\oplus\Field_{(-1)}$.
\end{corollary}

\begin{proof}
  When $\orL$ is the unknot, there is a genus one diagram
  with one generator, with $\delta$-grading $0$. This verifies the
  case where $n=1$. The case where $n>1$ follows now from
  Proposition~\ref{prop:KunnethDisjointUnion}.
\end{proof}

\begin{remark}
  Although we have used the holomorphic theory to prove
  Propposition~\ref{prop:KunnethDisjointUnion}, a proof purely within
  the context of grid diagrams can also be given as
  in~\cite[Section~\ref{GridBook:sec:AddingUnknots}]{GridBook}.
  Specifically, grid diagrams can be extended to give a slightly more
  economical description of unknotted, unlinked components.  Such
  components are represented by a square that is simultaneously marked
  with an $X$ and an $O$.  See Figure~\ref{fig:TwoComp} for an
  extended grid diagram for the two-component link, with two
  generators.  This picture can be used to easily verify
  Corollary~\ref{cor:unlink} when $n=2$.
\end{remark}
\begin{figure}[ht]
\centering
\includegraphics[width=3cm]{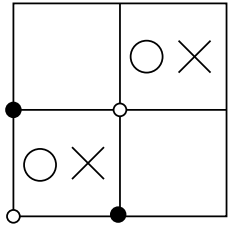}
\caption{{\bf Extended grid diagram of the two-component unlink.}
Simple computation determines the $\delta$-gradings of the two generators
shown by the diagram ({indicated} by the two full and the two
hollow dots, respectively).}
\label{fig:TwoComp}
\end{figure}

The following result will play an important role in {the subsequent
discussion}.  Given a $\Z$-graded chain complex $C$ over $\Field[U]$,
let $\Hom_{\Field [U]}(C,\Field[U])$ denote the chain complex of
$\Field[U]$-module homomorphisms $\phi\colon C\to \Field[U]$, graded
so that $\phi$ has degree $d$ if it sends the elements in $C_k$
to multiples of $U^{k+d}$.

\begin{prop}
  \label{prop:MirrorLink}
  If $\orL$ is an oriented link with $\ell$ components and $m(\orL)$
  is its mirror, then there is an isomorphism of graded chain
  complexes over $\Field[U]$:
  \[ \UCFK(m(\orL ))\cong 
\Hom_{\Field[U]}(\UCFK(\orL),\Field[U])\llbracket 1-\ell\rrbracket.\]
\end{prop}

\begin{proof}
  This follows from the corresponding duality under mirroring for link
  Floer homology; see~\cite[Proposition~8.3]{OSLinks}.
\end{proof}

From the universal coefficient theorem, it follows immediately that
for a knot
\begin{equation}
  \label{eq:mirror}
  \upsilon(m(K))=-\upsilon(K);
\end{equation}
see~\cite[Proposition~\ref{Concordance:prop:mirror}]{Upsilon} for a
more general version of this statement.
 
\section{The bound on the unknotting number}
\label{sec:unknotting}

Recall that $V$ is a two-dimensional $\Z$-vector space supported in grading $0$, so that if $M$ is a 
$\Z$-graded module over $\Field[U]$, then
\[ M\otimes_{\Field} V \cong M \oplus M\]
as $\Z$-graded modules over $\Field[U]$.
The key technical result in this section is the following:
\begin{prop}\label{prop:maps}
Suppose that $L_+, L_-$ are oriented links admitting projections which differ
only at one crossing, where the projection of $L_+$ is a positive
crossing, while for $L_-$ it is a negative crossing. Then there is
$n\in {\mathbb {N}}$ and there are $\Field [U]$-module maps
\[
N\colon \UHFL _* (L_+)\otimes _{\Field}V^{n}\to \UHFL _* (L_-) \otimes _{\Field}
V^n\qquad 
P \colon \UHFL _* (L_-)\otimes _{\Field} V^n\to \UHFL _{*-1} (L_+)\otimes _{\Field}
V^n , 
\]
such that $N$ preserves the $\delta$-grading, $P$ drops the
$\delta$-grading by one, and furthermore $P\circ N =U$ and $N\circ
P=U$. 
\end{prop}

\begin{remark} The same proposition holds without the stabilizing tensor
products with $V$; 
the tensor factors appear here since we choose to use grid diagrams.
\end{remark}

We postpone the proof of Proposition~\ref{prop:maps}, drawing first some of its immediate consequences.

\begin{theorem}\label{thm:changeofupsilon}
Suppose that $K=K_+$ is a given knot, together with
a projection and a distingushed positive crossing, and  let $K_-$ be the
knot we get by changing that crossing. Then,
\begin{equation}\label{eq:ineq}
0\leq \upsilon (K_-)-\upsilon (K_+)\leq 1.
\end{equation}
\end{theorem}

\begin{proof}
  Suppose that $x\in \UHFK (K_+)\otimes _{\Field}V^n$ is a generator which is
  non-torsion and has $\delta$-grading equal to $\upsilon (K_+)$. Then $N(x)$
  is also non-torsion (since $P(N(x))=Ux$ is non-torsion), therefore $\delta
  (N(x))\leq \upsilon (K_-)$. Since $N$ preserves $\delta$-grading, we get
  that $\upsilon (K_+)\leq \upsilon (K_-)$. Similarly, apply the map $P$ of
  Proposition~\ref{prop:maps} to a non-torsion element $y\in \UHFK
  (K_-)\otimes _{\Field} V^n$ of $\delta$-grading $\upsilon (K_-)$. Since $P$
  shifts degree by one, a simple modification of the above argument gives
  $\upsilon (K_-)-1\leq \upsilon (K_+)$. The two arguments give
  Inequality~\eqref{eq:ineq}.
\end{proof}

\begin{remark}
Note that a  more general version of this bound is proved in~\cite{Upsilon}, where it is shown that
\[
\Upsilon _{K_+}(t)\leq \Upsilon _{K_-}(t)\leq \Upsilon _{K_+}(t) +t
\]
holds for all $t\in [0,1]$. That proof appeals to the holomorphic
theory; the present proof is more in the spirit of our proof of
Theorem~\ref{thm:UnorientSlice}.  
\end{remark}

It follows immediately from
Theorem~\ref{thm:changeofupsilon} that $|\upsilon(K)|\leq u(K)$:
consider a minimal unknotting sequence of $K$, observe that $\upsilon$
for the unknot is $0$, and note that Theorem~\ref{thm:changeofupsilon}
shows that $\upsilon$ changes in absolute value by at most $1$ under
each crossing change. This bound will  be generalized in
Theorem~\ref{thm:OrientedSlice}.

Before turning to the proof of Proposition~\ref{prop:maps}, we give a further
consequence of it:

\begin{prop}
  \label{prop:StructureHFKinf}
  For any $\ell$-component link $\orL$, 
  $\UHFL(\orL)/\Tors\cong \bigoplus _1 ^r
  \Field[U]$, where $r=2^{\ell -1}$.
\end{prop}
\begin{proof}
Note that the maps induced by $N$ and $P$ on
$\UHFL(L)\otimes_{\Field[U]}\Field[U,U^{-1}]$ are isomorhpisms, since
both both $P\circ N$ and $N\circ P$ are invertible on
$\UHFL(L)\otimes_{\Field[U]}\Field [U, U^{-1}]$.  Considering a
sequence of crossing changes which turn a given link $\orL$ of $\ell$
components to the $\ell$-component unlink, and using
Corollary~\ref{cor:unlink}, we conclude that
\[ \UHFL(\orL)\otimes_{\Field[U]}\Field[U,U^{-1}]\cong  \Field[U,U^{-1}]^r.\]
The proposition now follows from the classification of 
finitely generated  modules over the principal ideal
domain $\Field[U]$, according to which
$(\UHFL
(L)/\Tors)\otimes _{\Field [U]}\Field [U, U^{-1}]\cong \UHFL (L)\otimes
_{\Field [U]}\Field [U, U^{-1}]$.
\end{proof}

Proposition~\ref{prop:StructureHFKinf} was used in the case where
$\ell=1$ to verify that $\upsilon$ is well-defined for knots. The
proposition also leads us to the natural extension of the
$\upsilon$-invariant of knots to the case of links.
\begin{definition}
  The {\defin{$\upsilon$-set of an oriented link $\orL$}} is a
  sequence of integers $\upsMin=\upsilon_1\leq \upsilon_2\leq\dots\leq
  \upsilon_{2^{\ell-1}}=\upsMax$ associated to $\orL$ as
  follows. Choose a set freely generating the quotient of the $\Field
  [U]$-module $\UHFL(L)$ by its torsion part (as an $\Field
  [U]$-module), with the property that each element is homogeneous
  with respect to the $\delta$-grading. Arrange the $\delta$-gradings
  of these homogeneous generators in order to obtain the
  $\upsilon$-set of $\orL$.

\end{definition}
It is easy to see that the above definition depends on the
$\Field[U]$-module structure of $\UHFL(\orL)$; i.e. it is independent
of the choice of the basis. By the invariance of unoriented link
homology, it follows that the $\upsilon$-set is an invariant of
$\orL$. (Compare also Corollary~\ref{cor:UnorientedUpsSet}.)

\begin{example}
  In general, we shall see in Lemma~\ref{lem:UpsNotTooFar} that for any oriented link, 
  $0\leq \upsMax(\orL)-\upsMin(\orL)\leq \ell-1$. It follows from Proposition~\ref{prop:KunnethDisjointUnion}
  that if $\orL$ is the  $\ell$-component unlink, $\upsMax(\orL)=0$ and $\upsMin(\orL)=1-\ell$;
  whereas by~\cite[Theorem~4.1]{AltKnots} (compare also~\cite[Theorem~\ref{Concordance:thm:AltKnots}]{Upsilon}),
  if $\orL$ is a link with connected, alternating projection, then
  $\upsMax(\orL)=\upsMin(\orL)=\frac{\sigma-\ell+1}{2}$.
\end{example}

Now we return to the proof of Proposition~\ref{prop:maps}.  We will
describe these maps in the grid context (explaining the presence of
the stabilizations in the statement).  By appropriately choosing the
grid diagram $\Grid _+$ representing $K_+$, it can be assumed that a
diagram $\Grid _-$ for $K_-$ is given by replacing the first column of
$\Grid_+$ with its second column (and vice versa), see the two
diagrams on the left of Figure~\ref{fig:crosschange}.  Indeed, these
diagrams can be drawn on the same torus, as  shown in the
diagram on the right of Figure~\ref{fig:crosschange}.
\begin{figure}
  \centering
  \input{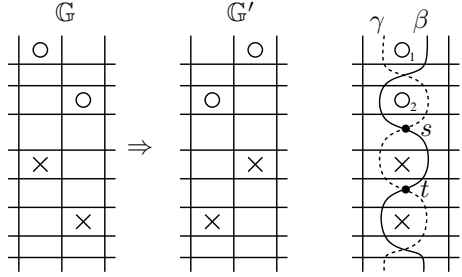}
  \caption{\textbf {Grid diagrams for crossing changes.}
\label{fig:crosschange}
On the left we show two distinguished columns in the diagram $\Grid _+$
representing $K_+$; switching these two columns gives the diagram $\Grid _-$,
shown in the middle, representing $K_-$. The diagram on the right represents both diagrams 
on the same torus, using  two intersecting curves $\beta$ and $\gamma$.}
\end{figure}
Notice that the two new curves $\beta$ and $\gamma$ define five
domains, four of which are bigons, each containing an $X$- or an
$O$-marking, while the fifth one contains all the other markings. The
two bigons containing the $X$-markings meet at $t \in \beta \cap
\gamma$, while the intersection of the two curves above the top
$X$-marking is $s$, cf. Figure~\ref{fig:crosschange}.

The maps $N$ and $P$ are defined by counting empty pentagons
(in the sense of 
\cite[Section~3.1]{MOST}). More precisely, suppose that
$\x_+$ is a generator of $\UGC (\Grid _+)$ and $\x _- $ is a generator
of $\UGC (\Grid _-)$. Then the $\Field [U]$-module
maps $N$ and $P$ on these chains are defined as
\[
  N (\x_+)=\sum_{\y _-\in\Gen(\Grid_-)}\sum_{p\in {\rm
      {Pent}}^0_s(\x_+,\y _-)} U^{\Weight(p)}\cdot \y _-,
\]
\[
 P (\x_-)=\sum_{\y _+\in\Gen(\Grid_+)}\sum_{p\in {\rm
     {Pent}}^0_t(\x_-,\y _+)} U^{\Weight(p)}\cdot \y _+,
\]
where ${\rm {Pent}}^0_s(\x _+, \y _-)$ and ${\rm {Pent}}^0_t(\x_-, \y
_+)$ denote the sets of empty pentagons with corner at $s$ and $t$,
respectively, connecting the indicated grid states. (The quantity
$\Weight(p)$ for an empty pentagon is defined as the corresponding
weight has been defined for rectangles: $\Weight(p)=\# p \cap (\Xs
\cup \Os )$.)

\bigskip

\begin{prooff} {\bf {of Proposition~\ref{prop:maps}.}}
Consider the module maps $N$ and $P$ defined above.  The usual
decomposition argument examining the interaction of rectangles
(contributing to the boundary maps of the chain complexes) and the
pentagons defining $P$ and $N$ (cf. \cite[Section~3.1]{MOST}) shows
that both maps are chain maps, inducing the maps (denoted by the same
symbols) of the proposition on the stabilized unoriented link Floer
homology groups. In a similar manner (by adapting the arguments of
\cite[Section~3.1]{MOST}) we can verify the claimed degree shifts.

To verify $N\circ P=U$ (and similarly, $P\circ N=U$) we construct maps
\begin{align*}
    H_+&\colon \UGC_{d}(\Grid_+)\to \UGC_{d}(\Grid_+) \\
    H_-&\colon \UGC_{d}(\Grid_-)\to \UGC_{d}(\Grid_-).
  \end{align*}
satisfying
  \begin{align}
    \partial \circ H_++ H_+\circ \partial &= P \circ N +
    U  \label{eq:Hkra1} \\
    \partial \circ H_-+ H_-\circ \partial &= N \circ P  + U , \label{eq:Hkra2}
    \end{align}
where $U$ denotes the operator of multiplication by $U$ in the
appropriate $\Field [U]$-module. Indeed, consider the set ${\rm
  {Hex}}^0_{t,s}(\x _+, \y _+)$ of empty hexagons (as in
\cite[Section~3.1]{MOST}) connecting the grid states $\x _+, \y _+\in
\Gen (\Grid _+)$, having two vertices at $t$ and $s$ (in this order). Define
${\rm {Hex}}^0_{s,t}(\x _-, \y _-)$ similarly (for grid states $\x _-,
\y _-$ of $\Grid _-$).  Then the definitions
\begin{align*}
H_+(\x _+)=\sum_{\y\in\Gen(\Grid_+)}\sum_{h\in {\rm {Hex}}^0_{t,s}
(\x,\y)} U^{\Weight(h)}\cdot \y _+ \\
H_-(\x _-)=\sum_{\y\in\Gen(\Grid_-)}\sum_{h\in {\rm {Hex}}^0_{s,t}
(\x,\y)} U^{\Weight(h)}\cdot \y _-\\
\end{align*}
provide the required maps.  Once again, the simple adaptation of
\cite[Section~3.1]{MOST} verifies the required identities of
Equations~\eqref{eq:Hkra1} and~\eqref{eq:Hkra2}. Indeed, by examining
the various decompositions of the composition of a hexagon (counted in
$H_{\pm}$) and a rectangle (counted in $\partial$), we either get an
alternate decomposition of the composite domain as a rectangle and a
hexagon, or the composition of two pentagons (counted in $P\circ N$ or
in $N\circ P$).  The only exception is the thin annular hexagon (containing no
complete circle, hence component in its interior) wrapping aroung the
torus. These domains do not admit alternate decompositions; on the
other hand, the position of the markings now implies that these
domains contain an $O$-marking, hence they provide an additive term of
multiplication by $U$, exactly as stated.
\end{prooff}

\renewcommand\lk{{\ell k}}
\section{Knot cobordisms}
\label{sec:knotcobordism}

Let $F$ be an embedded surface in $[0,1]\times S^3$, which meets
$\{0\}\times S^3$ and $\{1\}\times S^3$ in knots $K_0$ and $K_1$,
respectively.  The surface $F$ has an {\em Euler number} $e(F)$,
defined as follows. Fix the orientation on $[0,1]\times S^3$ we get by
concatenating the canonical orientation of $[0,1]$ with an orientation
of $S^3$. Take a local orientation system on $F$, and let $F'$ be a
small push-off of $F$, giving the Seifert framings of $K_0$ and $K_1$
in $\{0\}\times S^3$ and in $\{1\}\times S^3$, respectively.  A local
orientation system on $F'$ is induced by the given local orientation
system of $F$.  At each (transverse) intersection point $p\in
  F \cap F'$, compare the induced orientation from $T_p F\oplus T_p
F'$ with the orientation on $T_p ([0,1]\times S^3)$ and get a sign
$\pm 1$, called the \emph{local self-intersection number} at
$p$. Adding up these contributions at each intersection point gives
the Euler number $e(F)$. (Equivalently, pass to the orientable double
cover ${\widetilde F}$, pull back the normal bundle of $F$,
along with its trivialization at $\partial F$. Half of the relative
Euler number of this oriented $2$-plane bundle is the Euler number of
$F$.) When $F$ is orientable, the quantity defined in this manner
vanishes.
\begin{remark}\label{rem:orient}
Notice that if we turn the cobordism upside down, then we reverse the
orientation both on the $[0,1]$- and the $S^3$-factors, hence the
Euler number remains unchanged.  If the surface $F$ is embedded in
$S^3=\{ 1\}\times S^3$, we can make it a cobordism in two different
ways: we can push either end of the cobordism into $\{ 0\}\times S^3$
and keep the other one in $\{ 1\}\times S^3$. The resulting Euler
numbers of the two cobordisms will be opposites of each other: the two
presentations correspond to the two different orientations on $[0,1]$
(while keeping the orientation of $S^3$ unchanged).  Therefore, when
we consider an unorientable cobordism in $S^3$, its Euler
number makes sense only after we specify a direction on the cobordism,
that is, if we specify an incoming and an outgoing end of the surface
(viewed as a 2-dimensional cobordism in $[0,1]\times S^3$).
\end{remark}

A {\em saddle move} on a link $L$ is specified by an embedded
rectangle $B$ (which we call a ``band'') in $S^3$ with opposite sides
on $L$. A new link $L'$ is obtained by deleting the two sides of
the band in $L$ and replacing them with the other two sides of the
band.  Fix an orientation $\orL$ on $L$. A saddle move is called {\em oriented} if the orientations of
the two arcs in $L$ are compatible with the boundary orientation of
the band; otherwise, it is called an {\em unorientable saddle}.  An
unorientable saddle specifies a cobordism $F$ from $L$ to $L'$ with
$b_1(F,\Zmod{2})=1$. We will always apply the convention that, when
viewing the saddle band as a cobordism, the original link $L$ is in
$\{ 0\}\times S^3$ (that is, $L$ is the incoming end) and the
resulting link $L'$ is in $\{1 \} \times S^3$ (so $L'$ is the outgoing
end).

If $L=K$ is a knot, an unoriented saddle move gives rise to another knot $L'=K'$.
For an unorientable saddle, the relative Euler number can be computed
as follows.  

\begin{lemma} 
Suppose that $B$ is an unorientable saddle band from the knot $K$ to $K'$.
Choose a nonzero section $s$ of the normal bundle of the
band $B$, and choose a framing $\lambda$ of $K$ which agrees with $s$
along the two arcs in $B\cap K$.  Let $\lambda'$ be the induced framing of
$K'$. Then,
\begin{equation}\label{eq:EulerNumberAndLinking} 
e(B)=\lk(K,\lambda )-\lk(K',\lambda '),
\end{equation}
where here, for example, $\lk (K, \lambda )$ denotes the linking number of $K$
with the push-off of $K$ specified by the framing $\lambda$. (As before, $K$
is the incoming and $K'$ is the outgoing end of the cobordism.)  
\end{lemma}
\begin{proof}
From the definition of the linking number, it is clear that the Euler
number of the band is equal to the difference of the two linking
numbers. Indeed, by considering a nowhere zero section over $B$, the
difference of the linking numbers determines its difference from a
section with possible zeros, but which induces the Seifert framings at
the two ends.

The sign in the formula, however, deserves a short explanation. For
simplicity, assume that $K'$ bounds a surface $F$ with Euler number
$e(F)$. (The case of cobordisms follows along a similar logic.)  In
computing the Euler number consider a nonvanishing section of the
normal bundle along $F$ and consider the induced framing (still
denoted by $\lambda '$) along $K'$.  If we take the trivial cobordism
$W$, now from $K'=K'_0\subset \{ 0\}\times S^3$ to $K'_1\subset \{ 1\}
\times S^3$ with a section of the normal bundle which interpolates
between the framing $\lambda '$ on $K'_0$ and the Seifert framing on
$K_1'$, then this section will have zeros. Indeed, the (signed) number of zeros
is exactly the Euler number of the surface $F$ (since together with
the topologically trivial collar between $K'=K_0'$ and $K_1'$ and the
section there, we have a section inducing the Seifert framing). On the
other hand, the number of zeros along $W$ can be easily computed:
consider a Seifert surface of $K_0'$, push it into $D^4$ to get a
surface $W'$ and glue it to $W$. Extend the framing $\lambda '$ to a section
$\sigma$ of the normal bundle of $W'$.  Clearly, the number of zeros
of $\sigma$ is $\lk (K', \lambda ')$ (following from the definition of the
linking number), while if we glue $\sigma$ to our section over $W$ we
get a section of $W'\cup W$ inducing the Seifert framing on its
boundary, hence the sum of zeros of this section is zero. This shows
that over $W$ the signed number of zeros (and hence the Euler number
$e(F)$) is $-\lk (K', \lambda ')$, justifying the formula of
Equation~\eqref{eq:EulerNumberAndLinking}, and concluding the proof of
the lemma.  
\end{proof}
The above formula can be given in explicit terms if the saddle
band is related to an unoriented resolution of a crossing. Fix a
diagram ${\mathcal {D}}$ of a knot $K$, and choose a crossing in the
projection. Suppose that the unoriented resolution of that crossing
gives an unorientable saddle $B$ (embedded in $S^3$) that connects
$K$ to the result $K'$ of the resolution, see Figure~\ref{fig:resol}.
(Once again, we assume that, as a cobordism, $B$ is from $K$ to 
$K'$.)
\begin{figure}[ht]
\centering
\includegraphics[width=6cm]{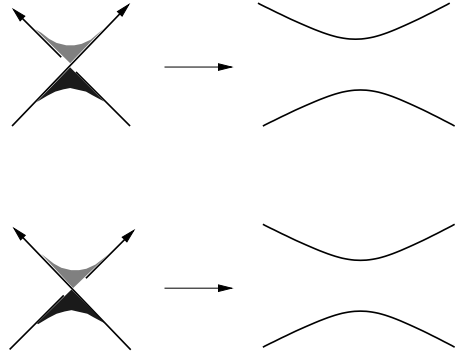}
\caption{{\bf Unorientable saddle band attachment at a crossing of a
    diagram.} In the top diagram the band resolves a positive
  crossing, in the bottom diagram it resolves a negative crossing.}
\label{fig:resol}
\end{figure}

Recall that the \emph{writhe} $\Wr ({\mathcal {D}})$ of the diagram
${\mathcal {D}}$ is defined as the sum of the signs of the
crossings. Alternatively, take $\lambda_{bb}$ to be the framing of $K$
given by the diagram (called the \emph{blackboard framing}): move each
point of the knot up (parallel to the projection) to get
$\lambda_{bb}$. Then $\Wr({\mathcal {D}})=\lk (K,
\lambda_{bb})$. The writhe $\Wr ({\mathcal {D}})$ (and
similarly $\lambda_{bb}$) depends on the chosen diagram; it is not an
invariant of $K$. On the other hand, for a 
projection ${\mathcal {D}}$ of a knot the writhe
$\Wr ({\mathcal {D}})$ is independent of the chosen orientation on the knot.

\begin{lemma}\label{lem:ComputeEulerNumber}
Let $K_1$ be a given knot, together with a diagram ${\mathcal {D}}_1$
and $B$ an unorientable saddle band
coming from an unoriented resolution of a crossing of ${\mathcal {D}}_1$.
Let $K_2$ denote the knot given by the resolution, together with 
the resulting diagram ${\mathcal {D}}_2$ of it.
Then,
  \begin{align*}
    e(B)=\Wr({\mathcal {D}}_1)-\Wr({\mathcal {D}}_2)+\epsilon,
  \end{align*}
  where 
  \begin{itemize}
    \item $\epsilon=+1$ if the resolution eliminates a positive crossing
    in ${\mathcal {D}}_1$
    \item $\epsilon=-1$ if the resolution eliminates a negative crossing
    in ${\mathcal {D}}_1$.
    \end{itemize}
\end{lemma}
\begin{proof}
Move the saddle band slightly up on the knot to achieve that it becomes
embedded in the plane, cf. Figure~\ref{fig:moveband}.
\begin{figure}[h]
\centering
\includegraphics[width=8cm]{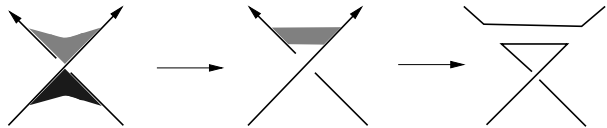}
\caption{{\bf Moving the saddle band.}  
By slighly pushing the band up on the knot, we can assume that it is
embedded by the projection in the plane.}
\label{fig:moveband}
\end{figure}
In this picture the vector field pointing upwards (parallel to the
projection) will give a nowhere zero vector field in the normal bundle
of the band $B$, restricting to two framings along $K_1$ and $K_2$.
The diagram for $K_1$ is still ${\mathcal {D}}_1$, but the diagram 
${\mathcal {D}}_2'$ we get for
$K_2$ is different from ${\mathcal {D}}_2$.
Since the chosen vector field induces the
blackboard framings on the two diagrams, the formula of
Equation~\eqref{eq:EulerNumberAndLinking} determines the Euler number
$e(B)$:
\begin{equation}\label{eq:Vesszos}
e(B)=\Wr ({\mathcal {D}}_1)-\Wr ({\mathcal {D}}_2').
\end{equation}

It is easy to see that the diagram ${\mathcal {D}}_2'$ for $K_2$ differs from
${\mathcal {D}}_2$ by a Reidemeister 1 move of introducing an extra crossing
(cf. the right-most diagram of Figure~\ref{fig:moveband}). Since the two
strands in this crossing were oriented so that after the resolution these
orientations are in conflict (since we consider the unoriented resolution), we
need to change the orientation on one of the strands, reversing the sign of
the crossing.  Hence $\Wr({\mathcal {D}}_2)=\Wr ({\mathcal {D}}_2')+\epsilon$,
which, combined with Equation~\eqref{eq:Vesszos} provides the result.
\end{proof}
\begin{remark}
In the same vein we can examine unorientable saddle band attachments
which create a new crossing in a diagram. The formula for
computing the Euler number is similar, with the rule that $\epsilon $
is equal to $-1$ if the saddle introduces a positive crossing in
${\mathcal {D}}_2$ and is $+1$ if it introduces a negative crossing in
${\mathcal {D}}_2$. The argument is essentially the same as the proof
given above.
\end{remark}

Our slice bounds in Section~\ref{sec:slicebounds} will depend on
``normal form'' theorems for knot cobordisms.  We will handle the
orientable and non-orientable situations slightly differently.  The
relevant theorem in the orientable case is from~\cite{KaShiSu} and its
non-orientable version is due to Kamada~\cite{Kamada}. To state these
in the form we will use later, recall that ${\mathcal
  U}_n(K)$ denotes the link obtained by adding $n$ unknotted, unlinked
components to a knot $K$.

\begin{theorem}$($Orientable normal form,~\cite{KaShiSu}$)$
  \label{thm:ONormalForm}
  Suppose that there is an orientable surface $F\subset [0,1]\times S^3$ of
  genus $g$, which is a cobordism from $K_1$ to $K_2$. Then, there are
  integers $c$ and $d$, and knots $K_1'$ and $K_2'$ with the following
  properties:
  \begin{itemize}
    \item $K_1'$ is gotten from ${\mathcal U}_c(K_1)$ by adding
      exactly $c$ orientable saddles.
    \item $K_2$ is gotten from ${\mathcal U}_d(K_2')$ by adding
      exactly $d$ orientable saddles.
    \item There is a cobordism $F'$ of genus $g$ from $K_1'$ to $K_2'$ which 
is composed by the addition of $2g$ orientable saddles. \qed
  \end{itemize}
\end{theorem}

\begin{theorem}$($Non-orientable normal form,~\cite{Kamada}$)$
  \label{thm:UNormalForm}
  Suppose that there is a non-orientable surface $F\subset [0,1]\times S^3$,
  which is a cobordism from $K_1$ to $K_2$. Then, there are integers $c$ and
  $d$, and knots $K_1'$ and $K_2'$ with the following properties:
  \begin{itemize}
    \item $K_1'$ is gotten from ${\mathcal U}_c(K_1)$ by adding
      exactly $c$ orientable saddles.
    \item $K_2$ is gotten from ${\mathcal U}_d(K_2')$ by adding
      exactly $d$ orientable saddles.
    \item There is a cobordism $F'$ from $K_1'$ to $K_2'$ composed of
      $b=b_1(F')=b_1(F)$ non-orientable saddles, and with
      $e(F')=e(F)$. \qed
  \end{itemize}
\end{theorem}
\begin{remark}
Although in \cite{Kamada} the normal form theorem is stated
for embedded, non-orientable \defin{closed} surfaces, the exact
same argument provides the result above for cobordisms between 
knots.
\end{remark}

  The outline of the proofs of the normal form theorems goes as
  follows: restrict the projection function $[0,1]\times S^3\to [0,1]$
  to the cobordims $F$. By generic position we can assume that the
  result is a Morse function, and it is easy to isotope $F$ so that
  (when increasing $t$ in $[0,1]$) we encounter first the index-0
  critical points, then the index-1 and finally the index-2 critical
  points. With a possible further isotopy we can arrange that index-1
  critical points correspond to the same value. By considering first
  those index-1 critical points for which the corresponding bands make
  the ascending disks of the index-0 handles and $K_1$
  connected (and repeating the same process for the 2-handles, now
  upside down), we get the desired form of the theorem.  Notice that
  in the non-orientable case the equality $e(F)=e(F')$ follows
  trivially from the fact that the subsurface given by the 0-handles
  and the orientable saddles is orientable, hence has vanishing Euler
  number.  In the non-orientable case further handle slides are needed
  to assure that all 1-handle attachments between the knots $K_1'$ and
  $K_2'$ are non-orientable. For more on
    Theorem~\ref{thm:ONormalForm} see \cite[Appendix~B.5]{GridBook}.

\newcommand\USaddle{\nu}
\newcommand\Split{\sigma}
\newcommand\Merge{\mu}
\section{Slice bounds from $\upsilon$}
\label{sec:slicebounds}

The proofs of the estimates on the genera of orientable and first
Betti numbers of non-orientable slice surfaces for a knot $K$ will
both rely on the normal form theorems of knot cobordisms discussed in
the previous section.  We start with the discussion of the orientable
case, and turn to the non-orientable case afterwards.

\subsection{Orientable slice bounds from $\upsilon$}
In order to prove the bound provided by $\upsilon (K)$ on the 
oriented slice genus $g_s(K)$ of $K$, we need to understand how
the invariant changes under oriented saddle moves. For this, the 
following proposition will be of crucial importance.

\begin{prop}
  \label{prop:OSaddleMove}
  Let $L$ and $L'$ be two links, related by an oriented saddle move,
  and suppose that $L'$ has one more component than $L$. Then, there is
  an integer $n\in {\mathbb {N}}$ and there are $\Field [U]$-module
  maps
  \begin{align*}
    \Split&\colon \UHFL(L)\otimes _{\Field}V^{n} \to \UHFL(L')\otimes
    _{\Field} V^{n-1} \\ \Merge &\colon \UHFL(L') \otimes
    _{\Field}V^{n-1}\to \UHFL(L) \otimes _{\Field} V^n
  \end{align*}
  with the following properties:
  \begin{itemize}
    \item $V$ is a two-dimensional $\Field$-vector space in $\delta$-grading 0,
    \item $\Split$ drops $\delta$-grading by one,
    \item $\Merge$ preserves $\delta$-grading,
    \item $\Merge\circ \Split$ is multiplication by $U$,
    \item $\Split \circ \Merge$ is multiplication by $U$.
  \end{itemize}
\end{prop}
The map $\Split$ will be referred to as the \emph{split map}
and $\Merge$ as the \emph{merge map}.
We prove the above proposition after establishing its key consequence:
\begin{theorem}
  \label{thm:SaddleMove}
  Let $L$ and $L'$ be two links which differ by an oriented saddle move, and
  suppose that $L'$ has one more component than $L$.
  Then,
  \begin{align}
    \upsMax(L) -1&\leq \upsMax(L')\leq \upsMax(L) \label{eq:SaddleIneqMax}\\
    \upsMin(L)-1 &\leq \upsMin(L')\leq \upsMin(L). \label{eq:SaddleIneqMin}
  \end{align}
\end{theorem}

\begin{proof}
  Consider a homogeneous non-torsion element $x\in\UHFL (L)\otimes
  V^n$ with maximal $\delta$-grading, i.e. $\delta (x )=\upsMax(L)$.
  By Proposition~\ref{prop:OSaddleMove}, its image $\Split(x)$ is
  non-torsion, and is of $\delta$-grading $\upsMax (L)-1$,
hence $\upsMax(L)-1\leq \upsMax(L')$. Similarly, if
  $y\in\UHFL (L')\otimes V^{n-1}$ is a non-torsion element with
  maximal $\delta$-grading $\upsMax(L')$, then its image $\Merge(y)$
  has $\delta$-grading $\upsMax(L')$, and it is non-torsion, so
  $\upsMax(L')\leq \upsMax(L)$, verifying
  Inequality~\eqref{eq:SaddleIneqMax}.

  Inequality~\eqref{eq:SaddleIneqMin} is obtained via a similar logic.
  The details, however, are slightly more involved, since the
  definition of $\upsMin (L)$ is not as straightforward as the
  definition of $\upsMax (L)$. Suppose that $a\in \UHFL (L)\otimes
  V^n$ is an element generating a free summand in $(\UHFL (L)/\Tors)
  \otimes V^n$ with $\delta$-grading $\upsMin (L)$. Then $\Split (a)$
  has $\delta$-grading $\upsMin (L)-1$, and it either generates a
  free summand in $(\UHFL (L')/\Tors)\otimes V^{n-1}$ or it is
  $U$-times such a generator. Indeed, if $\Split (a)=U^2h$ for some
  element $h$, then $\Merge (\Split (a))=Ua$ is equal to $U^2\Merge
  (h)$, and since multiplication by $U$ is injective on the factor
  $(\UHFL (L)/\Tors)\otimes V^{n}$, we would get $a=U\Merge (h)$,
  contradicting the choice of $a$ as a generator. Hence from the two
  possibilities (according to whether $\Split (a)$ is a generator, or
  $U$-times a generator) we get two inequalities, and 
  $\upsMin(L')\leq \upsMin(L)$ holds in both cases. With the same
  logic, starting now with a generator of $(\UHFL (L')/\Tors)\otimes
  V^{n-1}$ of $\delta$-grading $\upsMin (L')$ and applying $\Merge$,
  we get $\upsMin (L)\leq \upsMin (L')+1$, concluding the proof.
\end{proof}

We prove  Proposition~\ref{prop:OSaddleMove} using grid diagrams.

\bigskip

\begin{prooff}{\bf {of Proposition~\ref{prop:OSaddleMove}.}}
  It is not hard to see that any oriented band from $L$ to $L'$ can be
  represented by the following move: there is a grid diagram $\Grid $ for $L$
  such that by switching the $O$-markings in the first two columns (as shown
  by Figure~\ref{fig:SwappingOs}) we get the grid diagram $\Grid '$
  representing $L'$.  Let $n+\ell$ be equal to the grid index of $\Grid$ (and
  so of $\Grid '$), where $L$ has $\ell$ components (and so $L'$ has $\ell +1$
  components by our assumption). Let $O_1$ denote the $O$-marking in the
  first column of $\Grid$ and $O_2$ the $O$-marking in the second column of
  the same grid diagram.  After switching them, the new $O$-markings will be
  denoted by $O_1'$ and $O_2'$, respectively, see Figure~\ref{fig:SwappingOs}.
\begin{figure}
  \centering
\includegraphics[width=6cm]{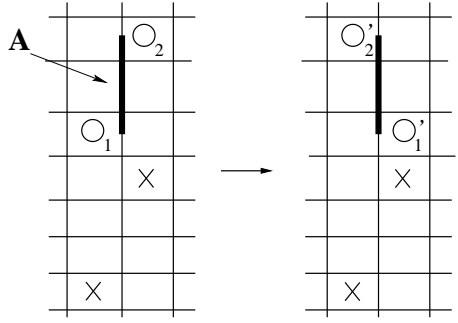}  
  \caption{\textbf {Grid diagrams for an oriented saddle
      move.}  \label{fig:SwappingOs} The left diagram is the first two
    columns of a grid diagram representing $L$, and we get a diagram
    for $L'$ by switching the $O$-markings of these columns (given on
    the right).  The figure also shows the arc ${\mathbf A}$ by the
    thicker segment. The complement of ${\mathbf A}$ in the circle is
    ${\mathbf {B}}$.}
\end{figure}

The grid states of $\Grid$ and of $\Grid '$ are naturally identified,
and can be classified into two types. This classification is based on
the position of the coordinate occupying the circle between the first
and second columns. Indeed, the two $O$-markings partition this circle
into two intervals, one of which (call it ${\mathbf {B}}$) passes by
the two $X$-markings, while the other one (which is, in some sense
'between the $O$-s') is called ${\mathbf {A}}$, see
Figure~\ref{fig:SwappingOs} where the interval ${\mathbf {A}}$ is
indicated. Now a grid state $\x$ is of type ${\mathbf {A}}$ if the
coordinate of $\x$ between the first and second columns is in
${\mathbf {A}}$; otherwise $\x$ is of type ${\mathbf {B}}$.

We define the $\Field [U]$-module maps $\Split \colon \UGC (\Grid )
\to \UGC (\Grid')$ and $ \Merge \colon \UGC (\Grid ') \to \UGC
(\Grid)$ as follows: for a grid state $\x \in {\mathbf A}$ consider
\[
\Split (\x)=U\cdot \x, \quad \Merge(\x )=\x,
\]
and for a grid state $\x \in {\mathbf B}$ take 
\[
\Split(\x )=\x, \qquad \Merge(\x )=U\cdot \x.
\]
The definition immediately implies that both $\Split\circ \Merge$ and
$\Merge \circ \Split$
are multiplications by $U$.
 
The proposition is proved once we show that the maps defined above on
the chain level are chain maps, which have the required behavior on
the $\delta$-grading.  Indeed, then the maps appearing in the
statement of the proposition will be the maps induced by these chain
maps on homology.

First we argue that the maps $\Split $ and $\Merge$ are chain maps;
below we will concentrate on the map $\Split$.  To this end, consider
a rectangle $r$ connecting two grid states $\x$ and $\y$ in
$\Grid$. Note that the $\Xs$-markings in $\Grid $ and in $\Grid '$
coincide, hence we only need to examine the change of interaction of
$r$ with $\Os$ and $\Os '$. If both grid states $\x, \y$ are from
${\mathbf B}$, then the rectangle $r$ contains $\{ O_1 , O_2\}$ with
the same multiplicity as it contains $\{ O_1', O_2'\}$, viewed as a
rectangle in either $\Grid$ or $\Grid '$.  The same holds if $\x$ and
$\y$ are both in ${\mathbf A}$.  If $\x\in {\mathbf A}$ and $\y\in
{\mathbf B}$, then the rectangle $r$, thought of as a rectangle in
$\Grid$, contains exactly one of $O_1$ or $O_2$, but it does not
contain either of $O_1'$ or $O_2'$, i.e. the contribution of $r$ to
$\partial \x$ contains $\y$ with an extra factor of $U$ not appearing
in the contribution of $r$ to $\partial' \x$. The definition of
$\Split$ compensates for this difference, verifying $\partial' \circ
\Split(\x)=\Split\circ \partial(\x)$ when $\x\in{\mathbf
  A}$. Similarly, in the case where $\x\in{\mathbf B}$ and
$\y\in{\mathbf A}$, $r$ contains neither of $O_1$ or $O_2$, but it
does contain exactly one of $O_1'$ or $O_2'$, so $r$ contributes an
extra $U$ factor in $\partial'(\x)$ which it does not in $\partial
(\x)$. This discrepancy is also compensated for in the definition of
$\Split$.  The map $\Merge$ is a chain map by the same logic.

In comparing the $\delta$-gradings of $\x$ in $\Grid$ and in $\Grid'$,
we first verify that for an element $\x \in {\mathbf A}$ we have $M_{\Os'}(\x
)=M_{\Os}(\x )+1$, while for $\x\in{\mathbf B}$,
$M_{\Os'}(\x)=M_{\Os}(\x)-1$.  Indeed, $\NEunnorm (\x , \x)$ is the
same in both diagrams, while (using Figure~\ref{fig:SwappingOs}) it is
easy to see that $\NEunnorm (\Os ', \Os ')= \NEunnorm (\Os ,
\Os)-1$. For the mixed terms $\NEunnorm (\Os ' , \x)= \NEunnorm (\Os ,
\x)$ and $\NEunnorm (\x , \Os ')= \NEunnorm (\x , \Os)$ for $\x \in
  {\mathbf {B}}$, while $\NEunnorm (\Os ' , \x)=\NEunnorm (\Os ,
  \x)-1$ and $\NEunnorm (\x , \Os ')=\NEunnorm (\x , \Os) -1$ for a
  grid state $\x \in {\mathbf {A}}$. Since $\Xs = \Xs '$, we get that
  $\delta _{\Grid }(\x )=\frac{1}{2}(M_{\Os }(\x ) +M_{\Xs } (\x
  ))+\frac{n-\ell}{2}$ and $\delta _{\Grid '}(\x )=\frac{1}{2}(M_{\Os
    '}(\x )+M_{\Xs }(\x ))+\frac{n-\ell-1}{2}$ are equal if $\x\in
  {\mathbf {A}}$ and $\delta _{\Grid '}(\x)= \delta _{\Grid }(\x )-1$
  if $\x \in {\mathbf {B}}$.  Since multiplication by $U$ drops
  $\delta$-grading by 1, from this if follows that $\delta _{\Grid
    '}(\Split (\x ))=\delta _{\Grid }(\x)-1$ and $\delta
  _{\Grid}(\Merge (\x ))=\delta _{\Grid '}(\x)$, as claimed.
\end{prooff}

With the above results at hand, now we can start examining the effect
of attaching an oriented band to a knot or link.  We start with the following
immediate corollary of Proposition~\ref{prop:KunnethDisjointUnion}:

\begin{lemma}
  \label{lem:TausForUnlinks}
  If $L$ is a link of the form $L={\mathcal U}_n(K)$ for some knot
  $K$, then $ \upsMax(L)=\upsilon(K)$ and
  $\upsMin(L)=\upsilon(K)-n$.  \qed
\end{lemma}

This result then implies the fact that adding $n$ saddles to
${\mathcal {U}}_n (K)$, the resulting knot will have
$\upsilon$-invariant equal to $\upsilon (K)$:

\begin{prop}
  \label{prop:Ribbon}
  If the knot $K_2$ is obtained from the link ${\mathcal U}_n
  (K_1)$ by adding $n$ saddles, then $\upsilon(K_1)=\upsilon(K_2)$.
\end{prop}

\begin{proof}
  Since $K_2$ is obtained from ${\mathcal U}_n(K_1)$ by applying $n$
  merge moves, from Theorem~\ref{thm:SaddleMove} it follows that
  \[ 
  \upsilon(K_2)=\upsMax(K_2) \geq \upsMax({\mathcal U}_n(K_1)) =
  \upsilon(K_1).\]
  Now the mirror $m(K_2)$ is also obtained from ${\mathcal U}_n
  (m(K_1))$ by adding $n$ saddles, so the same argument gives
  \[ \upsilon(m(K_2)) \geq \upsilon(m(K_1)).\]
  Equation~\eqref{eq:mirror} now allows us to turn these two inequalities
  to the statement of the proposition.
\end{proof}

Putting these together, we get a variant of 
the genus bound stated in Equation~\eqref{eq:slicebound}:
\begin{prop}\label{prop:cobversion}
Suppose that $F$ is an orientable, genus-$g$ cobordism in $[0,1]\times
S^3$ between the two knots $K_1$ and $K_2$. Then
\[
\vert \upsilon (K_1)-\upsilon (K_2)\vert \leq g.
\]
\end{prop}
\begin{proof}
 We apply the orientable normal form Theorem~\ref{thm:ONormalForm}.

 Using notation from that theorem, $F$ gives two
 knots $K_1'$ and $K_2'$ such that (according to
 Proposition~\ref{prop:Ribbon}) $\upsilon (K_1)=\upsilon (K_1')$ and
 $\upsilon (K_2)=\upsilon (K_2')$, and there is a cobordism $F'$
 between $K_1'$ and $K_2'$ of genus $g$ which decomposes as $2g$
 orientable saddle moves.  Order them so that each split move is
 followed by a merge move, hence we decompose $F'$ further as $G_1\cup
 \ldots \cup G_g$ such that each $G_i$ (between the knots $C_i$ and
 $C_{i+1}$) is a genus-1 cobordism composed by the addition of a split
 and a merge move. Applying Proposition~\ref{prop:OSaddleMove} to the
 subcobordisms $G_i$ we get that $\vert \upsilon (C_i)-\upsilon
 (C_{i+1})\vert \leq 1$, hence $\vert \upsilon (K_1')-\upsilon
 (K_2')\vert = \vert \upsilon (C_1)-\upsilon (C_{g+1})\vert \leq g$,
 concluding the argument.
\end{proof}
  
The slice genus bound of Equation~\eqref{eq:slicebound} now easily follows:
\begin{theorem}
  \label{thm:OrientedSlice}
For any knot $K\subset S^3$, 
$\vert \upsilon (K)\vert \leq g_s(K)$.
\end{theorem}
\begin{proof}
Suppose that $F\subset D^4$ is a slice surface of genus $g_s(K)$ for
the knot $K\subset S^3$. By deleting a small ball from $D^4$ with
center on $F$ it gives rise to a cobordism between the unknot
${\mathcal {O}}$ and $K$. Since the unknot has $\upsilon ({\mathcal
  {O}})=0$, the inequality of Proposition~\ref{prop:cobversion}
implies $\vert \upsilon (K)\vert \leq g_s(K)$. 
\end{proof}

\subsection{Non-orientable slice bounds from $\upsilon$}
Theorem~\ref{thm:UnorientSlice} will be proved using the following
analogue of Proposition~\ref{prop:OSaddleMove}.

\begin{prop}
  \label{prop:USaddleMove}
  Let $K$ and $K'$ be two knots which are related by an unorientable
  saddle move, with Euler number $e$.  Then, there is an 
integer $n\in {\mathbb {N}}$ and there are maps
  \[
  \USaddle\colon \UHFK (K) \otimes _{\Field}V^n\to
  \UHFK (K')\otimes _{\Field} V^n 
\qquad{\text{and}}\qquad
  \USaddle'\colon \UHFK(K')\otimes _{\Field}V^n \to
  \UHFK (K)\otimes _{\Field}V^n
  \]
  with the property that
\begin{itemize}
\item $V$ is a 2-dimensional $\Field$-vector space in $\delta$-grading 0,
\item $\USaddle$ drops  $\delta$-grading by $\frac{2-e}{4}$, i.e.,
for a homogeneous element $x\in \UHFK (K)\otimes _{\Field }V^n$ we have 
$\delta _{K'}(\USaddle (x))=\delta _{K }(x)-\frac{2-e}{4}$, 
\item $\USaddle'$ drops $\delta$-grading by $\frac{2+e}{4}$, i.e., for
  a homogeneous element $y\in \UHFK (K)\otimes _{\Field }V^n$ we have
  $\delta _{K }(\USaddle (y))=\delta _{K '}(y)-\frac{2+e}{4}$
\item $  \USaddle'\circ\USaddle = U$ and 
 $ \USaddle\circ\USaddle' = U$. 
\end{itemize}
\end{prop}

We turn to the proof of the above proposition after establishing a consequence:

\bigskip

\begin{prooff}{\bf {of Theorem~\ref{thm:UnorientSlice}.}}
Suppose that $F$ is a smooth cobordism from $K_1$ to $K_2$, and the
Euler number of $F$ is $e(F)$, while its first Betti number is
$b_1(F)$.  

If $F$ is orientable, then $e(F)=0$, the Betti number $b_1(F)$ is
equal to $2g(F)$, and the statement of the theorem follows from
Proposition~\ref{prop:cobversion}.

Suppose now that $F$ is non-orientable. According to the non-orientable
normal form Theorem~\ref{thm:UNormalForm}, there are knots $K_1'$ and
$K_2'$ and a cobordism $F'$ from $K_1'$ to $K_2'$ such that
$e(F')=e(F)$ and $b_1(F')=b_1(F)$.  Furthermore, by
Lemma~\ref{lem:TausForUnlinks} we have that $\upsilon (K_1)=\upsilon
(K_1')$ and $\upsilon (K_2)=\upsilon (K_2')$.  Therefore, in order to
prove the theorem, we need to prove it for $F'$, a cobordism built
from $b_1(F)$ unorientable saddle bands.

If there is a single unorientable saddle band between $K_1'$ and $K_2'$, then
Proposition~\ref{prop:USaddleMove} (with the roles of $K=K_1'$ and $K'=K_2'$)
provides the result.  {Indeed, applying the maps $\USaddle$ and
  $\USaddle'$ to non-torsion elements in the homology associated to $K_1'$ and
  $K_2'$ respectively and reasoning as in the proof of
  Theorem~\ref{thm:SaddleMove}}, we find that for a single unorientable
saddle move with Euler number $e$
\[
\upsilon (K_1')-\frac{2-e}{4}\leq \upsilon (K_2')\leq \upsilon (K_1')+
\frac{2+e}{4},
\]
implying 
\[
\abs{\upsilon (K_1')-\upsilon (K_2')+\frac{e}{4}}\leq \frac{1}{2}.
\]
Adding this for all the $b_1(F)$-many unorientable saddle moves (and
using the additivity of the Euler number $e$) we get the desired
inequality.
\end{prooff}

The proof of Proposition~\ref{prop:USaddleMove} will closely follow
the proof of Proposition~\ref{prop:OSaddleMove}. The maps will be
defined similarly, but computing the degree shifts is a little more
involved.  To this end, consider a grid diagram and fix a
fundamental domain for it, that is, consider the grid in the
plane. This extra choice naturally gives a a projection of the
knot. The writhe of this projection will be denoted by $\Wr (\Grid )$.
A further number can be associated to the planar grid as follows:

\begin{definition}
For a given planar grid diagram $\pGrid $ define the \defin{bridge
  index} $b(\pGrid)$ as the number of those markings which are local
maxima in the diagram for the antidiagonal height function.
\end{definition}
For a toroidal grid $\Grid$ with planar realization $\pGrid$, both
$\Wr (\pGrid )$ and $b(\pGrid )$ depend the choice of planar
realization.  According to the next lemma, which is an important
ingredient in the proof of Proposition~\ref{prop:USaddleMove}, their
difference gives a quantity which is an invariant of the toroidal
grid.  (For the statement, recall the definition of $\NESW$ from
Section~\ref{sec:defs}.)
\begin{lemma}\label{lem:wrbraid}
For a planar grid diagram $\pGrid$,
$\NESW (\Os - \Xs , \Os -\Xs )
=b (\pGrid )- \Wr (\pGrid )$.
\end{lemma}
\begin{proof}
  The projection corresponding to the planar grid diagram is composed of
  straight (vertical and horizontal) segments.  Let ${\mathcal {S}}_h$
    denote the $n$ horizontal, and ${\mathcal {S}}_v$  the $n$ vertical
    segments.  Each such segment $a\in {\mathcal {S}}_h \cup
    {\mathcal {S}}_v$ has a pair of markings $O(a)$ and $X(a)$ as its
  endpoints.  It is easy to see that
  \begin{align*}
    \NESW(\Os-\Xs,\Os-\Xs) &= \sum_{i,j}
    \NESW(\{O_i\},\{O_j\})-2\NESW(\{O_i\},\{X_j\})+\NESW(\{X_i\},\{X_j\})
    \\ &=  \sum _{a\in {\mathcal {S}}_h}\sum_{b \in {\mathcal S}_v}
    \NESW(\{O(a)\}-\{X(a)\},\{O(b)\}-\{X(b)\}).
  \end{align*}

 Let
 \[ z(a,b)=\NESW(\{O(a)\}-\{X(a)\},\{O(b)\}-\{X(b)\}).\] For $a\in
   {\mathcal S}_h$ and $b\in {\mathcal {S}}_v$, a simple case analysis can be
 used to compute $z(a,b)$.  If $a$ and $b$ are disjoint, then $z(a,b)=0$.  If
 $a$ and $b$ meet in an endpoint which is a local maximum or a local minimum
 for the antidiagonal height function, then $z(a,b)=\OneHalf$; if
 they meet in an endpoint which is neither, then $z(a,b)=0$. Finally, if $a$
 and $b$ intersect in an interior point, then $z(a,b)$ is the negative
   of the intersection number of $a$ and $b$; i.e. it is $\mp 1$ if the
   crossing of $a$ and $b$ has sign $\pm 1$.  Note that the number of local
 maxima equals the number of local minima (of the antidiagonal height
 function).
\end{proof}

The construction of the two maps encountered by
Proposition~\ref{prop:USaddleMove} follows closely the construction of the maps in
Proposition~\ref{prop:OSaddleMove}.  Let $\Grid$ be a grid
diagram representing the knot $K$. A grid diagram $\Grid'$ representing the result of
an unorientable saddle move $K'$ on $K$ can be described as follows.

Consider two distinguished columns of $\Grid$ and switch the position of
the $X$ in the first column and the $O$ in the second, that is, move
the $X$-marking of the first column to the second column (within its
row) and move the $O$-marking of the second column to the first column
(again, within its row), see the transition from the left-most to the
middle diagram of Figure~\ref{fig:nonor}. 
\begin{figure}[ht]
\centering
\includegraphics[width=8cm]{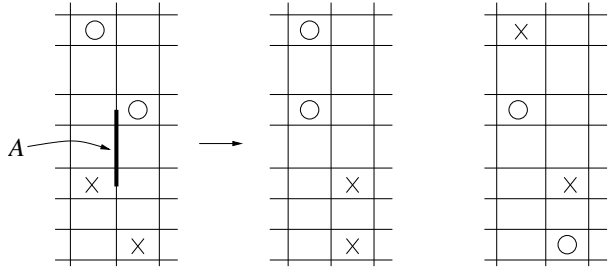}
\caption{{\bf Non-orientable saddle move in grids.}  We interchange
  the $X$- and the $O$-markings of the first and second columns of
  $\Grid$ (on the left), and get the illegal grid diagram
  $\Grid _{ill}$ in the middle. Starting at the bottom
  $X$-marking and traverse the knot until we get to the top
  $O$-marking, we change $X$ to $O$ and vice versa, to get the grid
  diagram $\Grid '$, the first two columns of which is shown on the
  right-most diagram.}
\label{fig:nonor}
\end{figure}
After this move, however, the result will not be a grid diagram
anymore: in the first column there are two $O$-markings, while in the
second column there are two $X$-markings. We call such a diagram
(where each row and each column has two markings in two different
squares, but the two markings are not necessarily distinct) an
\emph{illegal} grid.  Such a diagram still determines a knot (or
link), but does not specify an orientation on it.  Start at the bottom
$X$-marking in the second column and traverse through the knot (by
starting to move away from the other $X$-marking in the second
column), and change the $X$-markings to $O$'s and vice versa, until we
reach the top $O$ in the first column (and change it).  In this way
we restore a grid diagram $\Grid '$ which represents the knot $K'$
(with some orientation), cf. the right-most diagram of
Figure~\ref{fig:nonor}.

It is not hard to see that any unorientable saddle band attachment can
be achieved by this picture.  By fixing a planar presentation of
$\Grid$ and $\Grid '$, the grids also determine projections (hence
writhes) of the corresponding knots $K$ and $K'$, respectively.
Since with these conventions the switching of the markings corresponds
to the unoriented resolution of a positive crossing, for the Euler
number $e(B)$ of the saddle band $B$ (by
Lemma~\ref{lem:ComputeEulerNumber}) we have
\begin{equation}\label{eq:Ek}
e(B)=\Wr (\Grid )-\Wr (\Grid ')+1.
\end{equation}

The grid states of $\Grid$ and $\Grid '$ can be obviously identified
as before. Once again, we classify the grid states into two classes.
The circle between the first and the second column is partitioned into
two intervals by the $O$- and $X$-markings which we moved. Let
${\mathbf {A}}$ denote the interval which is not close to the further
two markings in the first two columns, and let ${\mathbf {B}}$ denote
the other interval (cf. Figure~\ref{fig:nonor} indicating ${\mathbf
  {A}}$). Correspondingly, the grid states with coordinate in
${\mathbf {A}}$ comprise the set ${\mathbf {A}}$, while the ones with
coordinate in ${\mathbf {B}}$ give ${\mathbf {B}}$.

The definition of the two $\Field [U]$-module maps follows the
corresponding definition of $\Split$ and $\Merge$ from
Proposition~\ref{prop:OSaddleMove}: for a grid state $\x \in {\mathbf
  {A}}$ we have
\begin{equation}\label{eq:nu}
\USaddle (\x )=U\cdot \x, \qquad \USaddle '(\x )= \x,
\end{equation}
and for a grid state $\x \in {\mathbf {B}}$ we have
\begin{equation}\label{eq:nuprime}
\USaddle (\x )=\x, \qquad \USaddle '(\x )=U\cdot \x,
\end{equation}
and obtain the maps $\USaddle \colon \UGC (\Grid ) \to \UGC (\Grid ')$
and $\USaddle ' \colon \UGC (\Grid ') \to \UGC (\Grid )$.  

\bigskip

\begin{prooff}{\bf {of Proposition~\ref{prop:USaddleMove}.}}
  Let us choose the grid diagrams $\Grid$ and $\Grid '$ given above
  (with $n+1$ being the common grid index), and define the two maps by
  the formulae of Equations~\eqref{eq:nu} and~\eqref{eq:nuprime}.  It
  is not hard to see that (just as in the oriented case) the maps are
  chain maps and their compositions (in any order) are multiplications
  by $U$.  Indeed, the same proof {from
    Proposition~\ref{prop:OSaddleMove}}, showing that $\Split$ and
  $\Merge$ are chain maps, applies here; since in unoriented knot Floer
  homology (as far as the boundary map goes) there is no distinction
  between the $X$- and $O$-markings.

Therefore all it remained to be verified are the formulae for the
degree shifts. Notice that although we only moved two markings (as we
did in the proof of Proposition~\ref{prop:OSaddleMove}), we also
relabeled a number of markings (by switching them from $X$ to $O$ or
conversely), possibly changing the $\delta$-grading significantly.
Let $\Grid_{ill}$ denote the intermediate illegal diagram we got by
swapping the $X$- and $O$-marking in the first two columns.  Although
$\Grid _{ill}$ is not a grid diagram, the terms $M_{\Os_{ill}}(\x)$
and $M_{\Xs _{ill}}(\x)$ (given by the adaptation of the formula of
Equation~\eqref{eq:MaslovDef}) make perfect sense for any grid state
$\x$, and indeed they can be easily related to $M_{\Os}(\x )$ and
$M_{\Xs}(\x)$ (giving the $\delta$-grading $\delta _{\Grid }$ in the
grid $\Grid$), just like in the proof of
Proposition~\ref{prop:OSaddleMove}.  A simple local calculation in the
first two columns of the grid $\Grid$ provides
\begin{align}
  \delta _{\Grid _{ill}}(\x )&=\left\{\begin{array}{ll}
               \delta _{\Grid}(\x )+1 & {\text{if $\x\in {\mathbf {A}}$}}\\ 
               \delta _{\Grid}(\x )&{\text{if $\x\in {\mathbf {B}}$.}}
\end{array}\right.
\end{align}

In the following we will concentrate on the degree shift of the map
$\USaddle$.  By the above formula, if $\USaddle _1(\x)$ denotes
$U\cdot \x $ or $\x$ in $\Grid _{ill}$ (depending on whether $\x$ in
$\Grid$ is in {\bf {A}} or in {\bf {B}}), then the above argument
shows that $\delta _{\Grid }(x )=\delta _{\Grid _{ill}}(\USaddle _1
(\x))$.

Therefore what is left to be done is to relate $\delta _{\Grid
  _{ill}}(\x)$ to $\delta _{\Grid '}(\x)$ for any grid state $\x$.
When writing down the defintions in the difference $\delta _{\Grid
  _{ill}}(\x) -\delta _{\Grid '}(\x)$, we realize that many terms
cancel. For example, the term $\NEunnorm (\x , \x)$ appears in both
(hence cancels in the difference). Furthermore, it is easy to see that 
\[
\NEunnorm (\x ,\Os_{ill}\cup \Xs _{ill})=
\NEunnorm (\x , \Os _{ill})+\NEunnorm (\x , \Xs _{ill})= \NEunnorm (\x , \Os
') + \NEunnorm (\x , \Xs ')=\NEunnorm (\x , \Os '\cup \Xs '),
\]
since in these sums we consider all the north-east pointing intervals from
coordinates of $\x$ to coordinates of $\Os_{ill}\cup \Xs _{ill}=\Os '\cup \Xs
'$. Similarly, 
\[
\NEunnorm (\Os _{ill}, \x )+\NEunnorm (\Xs _{ill} , \x )= \NEunnorm (\Os
', \x ) + \NEunnorm (\Xs ', \x),
\]
implying 
\begin{equation}\label{eq:sums}
\delta _{\Grid _{ill}}(\x )-\delta _{\Grid '}(\x )=\frac{1}{2}(
\NEunnorm (\Os _{ill}, \Os _{ill})+\NEunnorm (\Xs _{ill} , \Xs _{ill})- 
\NEunnorm (\Os ' , \Os ') - \NEunnorm (\Xs ' , \Xs ')).
\end{equation}


Partition $\Os _{ill}=\Os _1 \cup \Os _2$ and $\Xs _{ill}=\Xs _1 \cup \Xs _2$
in such a way that in getting $\Grid '$ we switch the markings in 
$\Os _2$ and $\Xs _2$: we have  $\Os '=\Os _1 \cup \Xs _2$ and
$\Xs '=\Xs _1 \cup \Os _2$. 
Now expanding Equation~\eqref{eq:sums} according to the above decompositions,
we get that 
\[
\delta _{\Grid _{ill}}(\x )-\delta _{\Grid '}(\x )=
\]
\[
\frac{1}{2}(
\NEunnorm (\Os _2, \Os _1)+\NEunnorm (\Os _1, \Os _2)+
\NEunnorm (\Xs _1 , \Xs _2)+ \NEunnorm (\Xs _2 , \Xs _1)-
\NEunnorm (\Os _1 , \Xs _2)-\NEunnorm (\Xs _2 , \Os _1)-
\NEunnorm (\Os _2 , \Xs _1)-\NEunnorm (\Xs _1 , \Os _2))= 
\]
\[
=\NESW (\Os _1-\Xs _1, \Os _2- \Xs _2). 
\]
Simple arithmetic shows that this quantity is equal to 
\[
\frac{1}{4}\big( \NESW (\Os _1 + \Os _2- \Xs _1-\Xs _2, 
\Os _1 + \Os _2- \Xs _1-\Xs _2)-
\NESW (\Os_1 +\Xs _2-\Os _2-\Xs _1, \Os_1 +\Xs _2-\Os _2-\Xs _1)\big)=
\]
\[
\frac{1}{4}\big( \NESW (\Os _{ill}- \Xs _{ill}, \Os _{ill}- \Xs _{ill})-
\NESW (\Os _{\Grid '}- \Xs _{\Grid '} , \Os _{\Grid '} - \Xs _{\Grid '})\big).
\]
Fix a planar presentation for both grids $\Grid _{ill}$ and $\Grid '$.
By Lemma~\ref{lem:wrbraid} we have that $\NESW (\Os _{\Grid '}- \Xs
_{\Grid '} , \Os _{\Grid '} - \Xs _{\Grid '}) =b(\Grid ')- \Wr (\Grid
')$. A simple local computation in the first two columns of $\Grid$
shows that $\NESW (\Os-\Xs , \Os - \Xs )+1= \NESW (\Os _{ill}-\Xs
_{ill}, \Os _{ill} - \Xs _{ill})$.  Local calculation in the first two
columns also implies that $b(\Grid )=b (\Grid ')$ (notice that the
quantity $b(\Grid)$ is insensitive of the change of markings from $X$
to $O$ or vice versa). Now
\[
\delta _{\Grid _{ill}}(\x )-\delta _{\Grid '}(\x )=\frac{1}{4}
\big(\Wr (\Grid ')- \Wr (\Grid ) +1\big)=-\frac{1}{4}\big( \Wr (\Grid )-\Wr
(\Grid') +1 -2)=-\frac{1}{4}(e(B)-2).
\]
In the last step we used the formula of Equation~\eqref{eq:Ek} (based on
Lemma~\ref{lem:ComputeEulerNumber}) expressing the Euler number of the
unorientable saddle in terms of the writhes.  Therefore $\delta _{\Grid
  '}(\USaddle (\x ) )=\delta _{\Grid}(\x )- \frac{2-e}{2}$, as
claimed.

Regarding the degree shift of the map $\USaddle '$ we can use the same
argument adapted to that situation, providing the claimed
result. Alternatively, the adaptation of the first part of this
argument shows that the map $\USaddle '$ shifts degree by a constant
(depending only on $\Grid '$ and $\Grid$); and we can easily determine
this constant knowing that the composition $\USaddle '\circ \USaddle$
is simply multiplication by $U$ on the chain complex, hence
it shifts degree by $-1$.  With this last observation the proof of
Proposition~\ref{prop:USaddleMove} (and therefore of
Theorem~\ref{thm:UnorientSlice}) is complete.
\end{prooff}

\newcommand\HFa{\widehat{HF}}
\section{Computations}
\label{sec:computations}

Computations of knot Floer homology can be used to calculate
$\upsilon(K)$ for several families of knots.  In
Section~\ref{sec:AltTorus} we state some results that
specialize computations from~\cite{Upsilon}.  Some of these examples
are then used in Section~\ref{sec:LinearIndep} to verify
Proposition~\ref{prop:Independence}. In Section~\ref{sec:Conway} we
show that $\upsilon$ vanishes for the Conway knots, whose slice status
is currently unknown.

\subsection{Alternating knots and torus knots}
\label{sec:AltTorus}

For any alternating knot $K$ (or more generally, any quasi-alternating
knot) we have $\upsilon(K)=\frac{\sigma(K)}{2}$;
see~\cite[Theorem~\ref{Concordance:thm:AltKnots}]{Upsilon}.
Similarly, as stated in Theorem~\ref{thm:TorusKnots}, a simple
algorithm determines $\upsilon$ of a torus knot (or more generally of
a knot which admits an $L$-space surgery) from its Alexander
polynomial.  Indeed, for such knots, the filtered chain homotopy type
of the complex for $\CFKm$ can be computed~\cite{NoteLens}, and this
computation can be used to determine $\Upsilon_K$ as
in~\cite{Upsilon}, and in particular $\upsilon$ (as stated in
Theorem~\ref{thm:TorusKnots}).

\begin{example}
For the $(3,4)$ torus knot $T_{3,4}$,
\[
\UHFK (T_{3,4})\cong \Field [U]_{(-2)}\oplus \big( \Field [U]/(U)\big)_{(-3)}
\oplus \big( \Field [U]/(U)\big)_{(-3)}.
\]
For comparison, $\HFKm (T_{3,4})=\Field [U]_{(-6,-3)}\oplus \big(
\Field [U]/(U)\big)_{(0,3)}\oplus \big( \Field
       [U]/(U^2)\big)_{(-2,0)}$, hence after collapsing the Maslov and
       Alexander gradings $M$ and $A$ to the $\delta$-grading $\delta
       = M-A$, we get $\HFKm (T_{3,4})=\Field [U]_{(-3)}\oplus \big(
       \Field [U]/(U)\big)_{(-3)}\oplus \big( \Field
              [U]/(U^2)\big)_{(-2)}$.
\end{example}

More generally, examining the Alexander polynomials of the family
$T_{3,q}$ of torus knots, it is easy to see that for $q\geq 1$,
\[ \upsilon (T_{3,q})=\left\{\begin{array}{ll}
    -\frac{2}{3}(q-1) & {\text{if $q\equiv 1\pmod{3}$}} \\
    -\frac{2}{3}(q-2)-1 &{\text{if $q\equiv 2 \pmod{3}$.}}
\end{array}\right.\]

\subsection{Linear independence}
\label{sec:LinearIndep}
We next turn to the verification that $\upsilon(K)$ is linearly independent
of $\tau$, $\delta$, $s$, and $\sigma$.  We will use the following facts about invariants of torus knots:
\begin{itemize}
  \item For a positive torus knot $K=T_{p,q}$, both $\tau(K)$ and
    $\frac{1}{2}s(K)$ are $\frac{(p-1)(q-1)}{2}$  (see~\cite{NoteLens} for
    $\tau$, and~\cite{RasmussenSlice} for $s$).
  \item If $p$ and $q$ are odd and relatively prime, the branched
    double cover of $T_{p,q}$ is the Brieskorn sphere $\Sigma(2,p,q)$;
    moreover, if $q=2pn\pm 1$ for some integer $n$, then
    $\Sigma(2,p,2pn\pm 1)=S^3_{\pm 1}(T_{2,p})$; and hence (using the
    formulae from~\cite{ManolescuOwens})
    \begin{align*}
      \delta(T_{p,2pn+1})&=0 \\
      \delta(T_{p,2pn-1})&=-2\lceil \frac{n}{2}\rceil.
    \end{align*}
      \item $\upsilon(T_{p,q})$ can be computed from
    \[ \Delta _{T_{p,q}}(t)= \frac{(t^{pq}-1)(t-1)}{(t^p-1)(t^q-1)}
    t^{-(\frac{pq-p -q-1}{2})},\] as in 
    Theorem~\ref{thm:TorusKnots}.
\end{itemize}

\begin{prooff}{\bf {of Proposition~\ref{prop:Independence}.}}
Using the signature calculations of \cite{MurasugiBook} and the above results,
we can now compute:
\vskip.2in
\begin{tabular}{r|r|r|r|r}
  &  $\delta/2$ & $\tau$ & $\upsilon$ & $\sigma/2$ \\
  \hline 
  $T_{3,5}$ & $-1$ & $4$ & $-3$ & $-4$ \\
  $T_{3,7}$ & $0$ & $6$ & $-4$ & $-4$ \\
  $T_{5,9}$ & $-1$  & $16$ & $-10$ & $-12$   \\
  $T_{5,11}$ & $0$ & $20$ & $-12$ & $-12$
\end{tabular} \newline
\vskip.2in
\noindent
The determinant of this $4\times 4$ matrix is non-zero. It follows
that the homomorphisms $\delta/2$, $\tau$, $\upsilon$, and $\sigma/2$
are linearly independent. (Moreover, it follows that the knots listed
above are linearly independent in the concordance group. This is not
surprising: according to \cite{Litherland}, all non-trivial torus
knots are linearly independent in the concordance group.)  Observe
that $2\tau-s=0$ for all torus knots; any knot $K$ with
$2\tau(K)\neq s(K)$ (the first examples of which were found by Hedden
and Ording~\cite{HeddenOrding}) now completes the linear independence
claim.
\end{prooff}

\subsection{Conway knots}
\label{sec:Conway}

It is an open problem, whether the Conway knot (cf. the left diagram
of Figure~\ref{fig:Conway}) is slice or not.  As we shall see soon,
$\upsilon$ cannot be used to settle this question.  
\begin{figure}[ht]
\centering
\includegraphics[width=12cm]{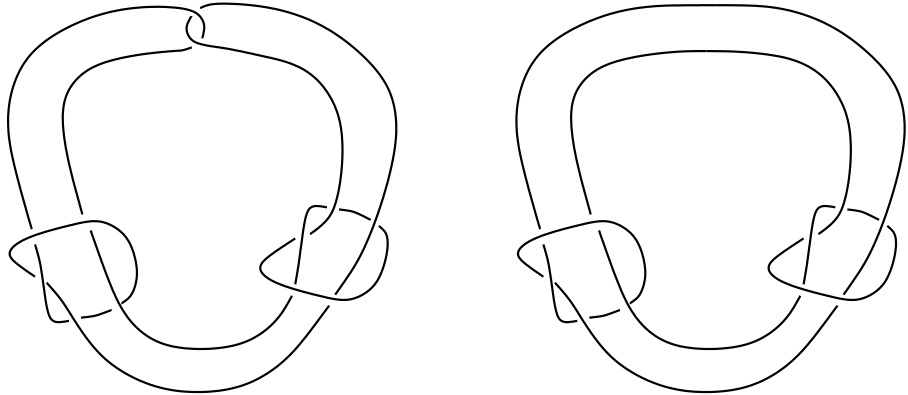}
\caption{{\bf The Conway knot and the Conway link.}
}
\label{fig:Conway}
\end{figure}

In fact, the Conway knot fits into an infinite family of knots
$C_{n,r}$, parameterized by two integers $r$ and $n$.  $C_{n,r}$ is
obtained by attaching a twisted band to the four-stranded
pretzel link $P(n+1,-n,-n-1,n)$ of
  Figure~\ref{fig:ConwayGen}; the parameter $r$ parameterizes the
number of full twists on the band, as shown on the right of
Figure~\ref{fig:ConwayGen}.  Thus, $C_{n,0}$ is the unknot for all
$n$, $C_{1,r}$ is the unknot for all $r$, and $C_{2,-1}$ is the Conway
knot $C$ from Figure~\ref{fig:Conway}. Notice that the pretzel link
$P(n+1, -n,-n-1,n)$ is isotopic to its mirror image: indeed, the
mirror is the pretzel link $P(-n-1, n,n+1,-n)$ which we get from the
original link by cyclically permuting the parameters ({which in turn
is straightforward to realize} by an isotopy).

\begin{figure}[ht]
\centering
\input{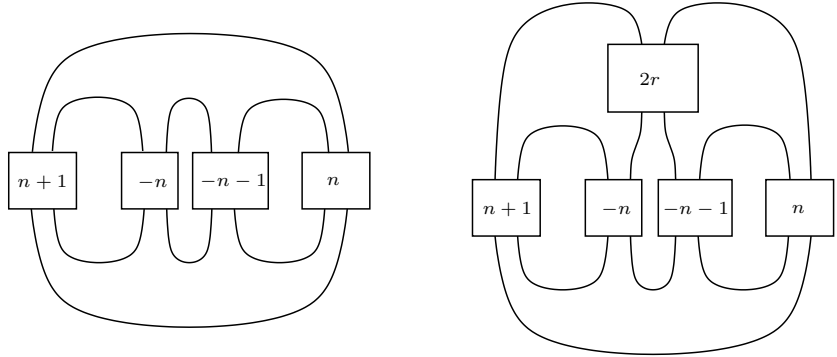}
\caption{{\bf The pretzel link
$P(n+1,-n, -n-1,n)$ and the two-parameter family of Conway knots.}
A box with $k\in \Z$ in it means $k$ right half-twists if $k\geq 0$ and $\vert k \vert $ left half-twists if $k<0$.
}
\label{fig:ConwayGen}
\end{figure}
\begin{prop}
  \label{prop:Conway}
  For all $n,r\in\Z$, the Conway knot $C_{n,r}$ has 
  $\upsilon(C_{n,r})=0$. 
\end{prop}

Before proving this result, we establish some general
principles.

\begin{lemma}
  \label{lem:UpsNotTooFar}
  If $\orL$ is an $\ell$-component link, then
  $\upsMax(\orL)-\upsMin(\orL)\leq \ell-1$.
\end{lemma}

\begin{proof}
  In $\ell-1$ oriented saddle moves, we can transform $\orL$ into a knot $K$.
  Applying Theorem~\ref{thm:SaddleMove} $\ell-1$ times, we get
  \[ \upsilon(K) -\ell+1 \leq \upsMax(\orL) \leq \upsilon(K)~\text{and}~
  \upsilon(K) -\ell+1 \leq \upsMin(\orL) \leq \upsilon(K),\]
  so the lemma follows.
\end{proof}

\begin{lemma}
  \label{lem:SelfMirrorLink}
  Let $\orL$ be a two-component link with the property that $m(\orL)=\orL$.
  Then, $\upsMin(\orL)=-1$ and $\upsMax(\orL)=0$.
\end{lemma}

\begin{proof}
  It follows from Proposition~\ref{prop:MirrorLink} that
  $\upsMax(\orL)=\upsMax(m(\orL))=-\upsMin(\orL)-1$; i.e.  the
  $\upsilon$-set of $\orL$ is of the form $(-c-1,c)$ with
  $-c-1\leq c$.  Lemma~\ref{lem:UpsNotTooFar} gives the
  inequality $2c+1\leq 1$, and so $c=0$.
\end{proof}

\begin{prooff}{\bf of Proposition~\ref{prop:Conway}.}
  Each  Conway knot is obtained by  adding an oriented band to the 
  $(-n-1,n,n+1,-n)$ pretzel link $\orL$. Since $m(\orL)=\orL$, 
  Lemma~\ref{lem:SelfMirrorLink} shows that 
  its $\upsilon$-set is $(-1,0)$. 
  Since $C_{n,r}$ is obtained from $\orL$ by a single oriented saddle move, 
  we can apply both inequalities from 
  Theorem~\ref{thm:SaddleMove} to conclude that $0=\upsilon(C_{r,n})$.
\end{prooff}

It is natural to wonder if $\upsilon$ remains invariant under Conway
mutation.  Note that there is a two-parameter family of slice
knots (and so with $\upsilon=0$), the {\em Kinoshita-Terasaka knots}
$KT_{n,r}$, which differ from the $C_{n,r}$ by a Conway mutation.

\section{Unoriented link invariants}
\label{sec:Orientations}

Using Lemma~\ref{lem:wrbraid}, we can modify our earlier constructions to define an invariant of unoriented links,
as follows.

\begin{prop}
  \label{prop:UnorientedLinkInvariant}
  Let $L$ be a link and $\orL$ be an orientation on it.
  The  $\OneHalf \Z $-graded group
  $\UHFL(\orL)\llbracket \frac{\sigma(\orL)-\ell+1}{2} \rrbracket$
  is independent of the choice of orientation on $L$.
\end{prop}

\begin{proof}
  Fix an orientation $\orL$ on $L$, and let $\Grid$ be a grid diagram
  representing $\orL$.  A grid diagram representing $L$ with any other
  orientation is obtained from $\Grid$ by exchanging some $O$- and
  $X$-markings. Let $\Grid'$ be another grid diagram so obtained.  Let
  $\Os$ and $\Xs$ be the markings in $\Grid$ and $\Os'$ and $\Xs'$ be
  the markings in $\Grid'$.  Let $G$ and $G'$ be two planar
  realizations of $\Grid$ and $\Grid'$ using the same fundamental
  domain in the torus. We can think of the $\delta$-grading from
  $\Grid$ and the one from $\Grid'$ as defining two functions
  $\delta\colon \Gen(\Grid)\to\Z$ and $\delta'\colon
  \Gen(\Grid)\to\Z$.  By bilinearity, for any $\x\in\Gen(\Grid)$,
  \begin{align}
    \delta(\x)-\delta'(\x)&=\frac{1}{2}(M_{\Os}(\x)+M_{\Xs}(\x)-M_{\Os'}(\x)-M_{\Xs'}(\x)) \nonumber \\
    &= \frac{1}{2}(\NESW(\Os,\Os)+\NESW(\Xs,\Xs)-\NESW(\Os',\Os')-\NESW(\Xs',\Xs')) \label{eq:d1} \\
    &= \frac{1}{4}(\NESW(\Os-\Xs,\Os-\Xs) -\NESW(\Os'-\Xs',\Os'-\Xs'))
  \end{align}
  By Lemma~\ref{lem:wrbraid}, since $b(G)=b(G')$, it follows that
  \begin{equation}
    \label{eq:d2}
    \NESW(\Os-\Xs,\Os-\Xs) -\NESW(\Os'-\Xs',\Os'-\Xs')=\writhe(G')-\writhe(G).
  \end{equation}
  Writing $\orL=\orL_1\cup\orL_2$ and $\orL'=-\orL_1\cup\orL_2$, it is obvious that
  \begin{equation}
    \label{eq:d3}
    \writhe(G')-\writhe(G)=4\lk(\orL_1,\orL_2).
  \end{equation}
  It is a straightforward consequence of the Gordon-Litherlan
  formula from~\cite{GordonLitherland} that
  \begin{equation}
    \label{eq:d4}
    \lk(\orL_1,\orL_2)=\frac{1}{2} (\sigma((-\orL_1)\cup \orL_2)-\sigma(\orL_1\cup\orL_2).
  \end{equation}
  Putting together Equations~\eqref{eq:d1},~\eqref{eq:d2},~\eqref{eq:d3}, and~\eqref{eq:d4}, we conclude that
  \[ \delta(\x)+\frac{\sigma(\orL)}{2}=\delta'(\x)+\frac{\sigma(\orL')}{2};\]
  the statement follows.
\end{proof}

\begin{definition}
  \label{def:UnorientedLinkSet}
Let $L$ be an oriented $\ell$-component link, and choose an
orientation $\orL$ on $L$.  The \defin{renormalized $\upsilon$-set of
  $L$} is the sequence of possibly half-integers $\upsilon_1'\leq
\upsilon_2'\leq\dots\leq \upsilon_{2^{\ell-1}}'$ defined by
$\upsilon'_i=\upsilon_i-\frac{\sigma-\ell+1}{2}$, where
$\upsilon_1\leq\dots\leq\upsilon_{2^{\ell-1}}$ is the $\upsilon$-set
of $\orL$, and $\sigma$ is the signature of $\orL$.
\end{definition}

The following is an immediate consequence of
Proposition~\ref{prop:UnorientedLinkInvariant}:

\begin{corollary}
  \label{cor:UnorientedUpsSet} 
  The unoriented link set of $L$ is an unoriented link invariant. \qed
\end{corollary}

By Proposition~\ref{prop:MirrorLink}, if
$\{\upsilon_i'\}_{i=1}^{2^{\ell-1}}$ is the renormalized
$\upsilon$-set of $L$, then
$\{-\upsilon'_{2^{\ell-1}-i+1}\}_{i=1}^{2^{\ell -1}}$ is the
renormalized $\upsilon$-set of its mirror.

For an alternating, $\ell$-component link with connected projection,
the renormalized $\upsilon$-set is the number zero, taken with
multiplicity $2^{\ell-1}$.

\bibliographystyle{plain}
\bibliography{biblio}

\end{document}